\documentclass[12pt,
	]{amsart}
\usepackage[margin=1in
	]{geometry}
\usepackage{graphicx}
\usepackage{amssymb}
\usepackage{setspace}
\usepackage{hyperref}
\newtheorem{introtheorem}{Theorem}
\newtheorem{theorem}{Theorem}[section]
\newtheorem{proposition}[theorem]{Proposition}
\newtheorem{corollary}[theorem]{Corollary}
\newtheorem{lemma}[theorem]{Lemma}
\newtheorem*{theorem*}{Theorem}
\theoremstyle{definition}
\newtheorem{remark}[theorem]{Remark}
\renewcommand{\bar}{\overline}
\renewcommand{\tilde}{\widetilde}
\renewcommand{\hat}{\widehat}
\renewcommand{\emptyset}{\varnothing}
\newcommand{\mrm}[1]{\mathrm{#1}}
\newcommand{\mbb}[1]{\mathbb{#1}}

\newcommand{\mcal}[1]{\mathcal{#1}}
\newcommand{\mfrak}[1]{\mathfrak{#1}}
\newcommand{\lp}{\left(}\newcommand{\rp}{\right)}
\newcommand{\define}[1]{{\textit{\textbf{#1}}}}
\newcommand{\sub}[1]{\mrm{sub}\lp#1\rp}
\newcommand{\class}[1]{\left[#1\right]}
\newcommand{\bmat}{\lp\begin{smallmatrix}}\newcommand{\emat}{\end{smallmatrix}\rp}
\newcommand{\unbased}[1]{#1^\circ}
\newcommand{\inv}[1]{#1^{-1}}
\newcommand{\rev}[1]{\bar{#1}}
\newcommand{\proj}[1]{\hat{#1}}
\newcommand{\lift}[1]{\tilde{#1}}
\newcommand{\Z}{\mbb{Z}}
\newcommand{\nnegZ}{\mbb{N}}
\newcommand{\R}{\mbb{R}}
\newcommand{\threereal}{\R^3}
\newcommand{\tworeal}{\R^2}
\newcommand{\C}{\mbb{C}}
\newcommand{\sC}{\C^\ast}
\newcommand{\nsl}{\mrm{SL}(n)}
\newcommand{\twosl}{\mrm{SL}(2)}
\newcommand{\pngl}{\mrm{PGL}(n)}
\newcommand{\ptwogl}{\mrm{PGL}(2)}
\newcommand{\sk}{\mcal{S}_n^q}
\newcommand{\usk}{\unbased{\lp\sk\rp}}
\newcommand{\ann}{\mfrak{A}}
\newcommand{\man}{\mfrak{M}}
\newcommand{\surf}{\mfrak{S}}
\newcommand{\intt}{\mfrak{I}}
\newcommand{\badset}{Q}
\newcommand{\altbadset}{Q^\prime}
\newcommand{\sym}{\mrm{Sym}}
\newcommand{\qroot}{q^{1/n}}
\newcommand{\braid}{\mrm{Brd}}
\newcommand{\gen}{\mcal{G}}
\newcommand{\indeq}{\overset{N}{=}}
\newcommand{\altindeq}{\overset{N^\prime}{=}}
\newcommand{\basgen}{\mcal{B}}
\newcommand{\tor}{\mfrak{T}}
\newcommand{\ssk}{\tilde{\mcal{S}}_n^q}
\newcommand{\punc}{\mfrak{P}}
\newcommand{\monoangle}{\mfrak{D}_1}
\newcommand{\biang}{\mfrak{D}_2}
\newcommand{\ttriang}{\mfrak{D}_3}
\newcommand{\source}{\unbased{W}_\mrm{source}}
\newcommand{\sink}{\unbased{W}_\mrm{sink}}
\newcommand{\idealtriang}{\mfrak{L}}
\newcommand{\splitidealtriang}{\hat{\idealtriang}}
\newcommand{\comp}{\mfrak{K}}
\newcommand{\wroot}{\omega}
\newcommand{\qtorus}[1]{\mcal{T}_n^\wroot(#1)}
\newcommand{\qtr}{\tau^q_\idealtriang}
\newcommand{\discrete}{\Theta_n}
\newcommand{\cornerlessdiscrete}{\discrete^\prime}
\newcommand{\cltorus}[1]{\mcal{T}_n^1(#1)}
\newcommand{\elemedgemat}[1]{S^\mrm{edge}_{#1}}
\newcommand{\elemleftmat}[1]{S^\mrm{left}_{#1}}
\newcommand{\elemrightmat}[1]{S^\mrm{right}_{#1}}
\newcommand{\edgemat}{M^\mrm{edge}}
\newcommand{\leftmat}{M^\mrm{left}}
\newcommand{\rightmat}{M^\mrm{right}}
\newcommand{\qleftmat}{L^\wroot}
\newcommand{\qrightmat}{R^\wroot}
\newcommand{\incl}{\iota}
\newcommand{\counit}{\epsilon}
\newcommand{\degree}{\mrm{deg}}
\newcommand{\fan}{\mcal{C}}
\title[Web basis for the annulus]{Monomial web basis for the SL(N) skein algebra of the twice punctured sphere}
\author[T. Cremaschi]{Tommaso Cremaschi}
\address{School of Mathematics, Trinity College Dublin, 17 Westland Row, Dublin 2, Ireland}
\email{cremasct@tcd.ie}
\author[D. C. Douglas]{Daniel C. Douglas}
\address{Department of Mathematics, Virginia Tech, 225 Stanger Street, Blacksburg, VA 24061, USA}
\email{dcdouglas@vt.edu}
\date{\today}
\thanks{The first author was partially supported by 101107744–DefHyp. The second author was partially supported by the Simons Foundation grant 327929.}
\begin{document}
\begin{abstract}
We give a new proof of a slightly modified version of a result of Queffelec--Rose, by constructing a linear basis for the $\nsl$ skein algebra of the twice punctured sphere for any non-zero complex number $q$, excluding finitely many roots of unity of small order.  In particular, the skein algebra is a commutative polynomial algebra in $n-1$ generators, where each generator is represented by an explicit $\nsl$ web, without crossings, on the surface.  This includes the case $q=1$, where the skein algebra is identified with the coordinate ring of the $\nsl$ character variety of the twice punctured sphere.  The proof of both the spanning and linear independence properties of the basis depends on the so-called $\nsl$ quantum trace map, due originally to Bonahon--Wong in the case $n=2$.  Two consequences of our method are that the quantum trace map and the so-called splitting map embed the polynomial algebra into the Fock--Goncharov quantum higher Teichm\"uller space and the L\^{e}--Sikora stated skein algebra, respectively, of the annulus.  We end by discussing the relationship with Fock--Goncharov duality.  
\end{abstract}
\maketitle

To a closed oriented three-dimensional manifold $\man$ containing a, possibly empty, framed oriented knot, or link, $K$ colored by a representation of the Lie group $\mrm{SU}(n)$, and to a natural number $k$ called the quantization parameter or level, Witten \cite{WittenCommMathPhys89} and later Reshetikhin--Turaev \cite{reshetikhinMR1091619}  assigned a numerical quantum invariant $Z$ now called the WRT invariant.  This number $Z$ is expressed in terms of a root of unity $q$ of order approximately $n+k$.  The WRT invariant satisfies certain local relations, called skein relations, and when $\man=\mfrak{S}^3$ is the sphere $Z$ essentially recovers the Jones--Reshetikhin--Turaev $\mrm{SU}(n)$ knot polynomial associated to the colored knot $K$.  For a general manifold $\man$, the computation of $Z$ bootstraps from the sphere case by using the fact that any such $\man$ can be obtained by integral Dehn surgery on a link, equivalently linked solid tori, in the sphere, and not to be confused with the link $K$.  

Witten's definition of $Z$ is by a physical path integral in Chern--Simons theory, very roughly speaking, integrating over all gauge equivalence classes of $\mrm{SU}(n)$ connections on $\man$ the trace of the monodromy of the knot in the representation.  In the classical limit as $k$ goes to infinity, $Z$ concentrates on the critical points of the integrand, namely the flat connections up to equivalence, which are the same as group homomorphisms from the fundamental group $\pi_1(\man)$ to $\mrm{SU}(n)$ considered up to conjugation.  Another name for this space of critical points is the $\mrm{SU}(n)$ character variety.  Character varieties are important objects of study in low-dimensional geometry and topology \cite{fockMR2233852, GoldmanInvent86, gunninghamMR4557403, HitchinTopology92, LabourieInvent06}.  The current paper is primarily concerned with studying a quantization of the character variety of the solid torus, equivalently the thickened annulus $\ann\times(0,1)$, called the $\mrm{SU}(n)$ skein algebra $\sk(\ann)$.  Elements of the skein algebra are represented as formal linear combinations of links in the solid torus, considered up to local skein relations.    

Reshetikhin--Turaev, also \cite{turaevMR1217386}, put Witten's construction on a more solid mathematical footing by defining the WRT invariant in terms of the representation theory of the quantum group $\mrm{SU}(n)_q$.  Lickorish \cite{lickorishMR1227009} formulated a more elementary and combinatorial approach to computing the $n=2$ invariant by purely skein theoretic methods.  One fundamental property employed to this end is that the skein algebra $\mcal{S}_2^q(\ann)$ of the annulus is isomorphic to the polynomial algebra $\C[x]$ in one variable, allowing for geometrically rich families of polynomials, such as Chebyshev polynomials, to be represented and manipulated pictorially.  Analogous skein constructions for the $n=3$ and  $n\geq2$, sometimes written $n=n$, invariants were developed in \cite{ohtsukiMR1457194} and \cite{yokotaMR1427678}, respectively, where the former took advantage of the property that the skein algebra $\mcal{S}_3^q(\ann)$ is isomorphic to the two variable polynomial algebra $\C[x,y]$.  

The primary contribution of the present paper is a new proof of a slightly modified version of the following result.

\begin{theorem*}[Queffelec--Rose \cite{queffelecMR3729501}]
The $\mrm{SU}(n)$ skein algebra of the annulus is isomorphic to a polynomial algebra $\C[x_1,x_2,\dots,x_{n-1}]$ in $n-1$ generators.  
\end{theorem*}

More precisely, we prove the following result.

\begin{introtheorem}\label{thm:intromaintheorem}
For all non-zero complex numbers $q$ except for finitely many roots of unity of small order approximately less than $n$, but allowing $q=1$, the skein algebra $\sk(\ann)$ of the annulus is isomorphic to a polynomial algebra $\C[x_1,x_2,\dots,x_{n-1}]$ in $n-1$ generators.  
\end{introtheorem}

See Theorem \ref{thm:maintheorem} for the most precise statement.  

To be more clear, there are multiple versions of the $\mrm{SU}(n)$, or $\nsl$, skein algebra, including Sikora's version described in Theorem \ref{thm:intromaintheorem}, some of which are briefly discussed now.  Based on knotted loops, the HOMFLYPT \cite{FreydBullAmerMathSoc85, Przytycki87ProcAmercMathSoc} and $n=2$ Kauffman \cite{KauffmanTopology87} skein algebras $\mcal{S}(\surf)_\mathrm{HOMFLYPT}$ and $\mcal{S}_2^q(\surf)$ for thickened surfaces $\surf\times(0,1)$, and more generally skein modules for three-dimensional manifolds $\man$, were initially studied by Turaev \cite{turaevMR0964255} and Przytycki \cite{PrzytyckiBullPolishAcad91}.  The, in general non-commutative, algebra structure is by superposition relative to the surface.  Kuperberg \cite{KuperbergCommMathPhys96} introduced the $n=3$ skein algebra $\mcal{S}_3^q(\surf)$ based on certain trivalent graphs, called $3$-webs, in addition to loops.  Sikora \cite{SikoraAlgGeomTop05} generalized Kuperberg's algebra to the $n=n$ skein algebra $\sk(\surf)$ by incorporating $n$-valent webs.  The skein algebra $\mcal{S}^q_n(\surf)_\mathrm{CKM}$ of Cautis--Kamnitzer--Morrison \cite{CautisMathAnn14} is based on colored trivalent graphs, and is the algebra described in the above theorem by Queffelec--Rose.  L\^{e}--Sikora's \cite{leMR3827810, LeArxiv21} stated skein algebra $\mcal{S}^q_n(\surf)_\mathrm{st}$ is similar to Sikora's skein algebra $\sk(\surf)$, but is sensitive to when the surface has nonempty boundary.  In particular, the skein algebra $\sk(\ann)$ of the closed annulus equals that of the open annulus, equivalently the twice punctured sphere, but the associated stated skein algebras are different.  

Note it is not claimed in Theorem \ref{thm:intromaintheorem} that the Sikora skein algebra $\sk(\ann)$ is not a polynomial algebra for the excluded roots of unity $q$, and it is indeed true that $\sk(\ann)$ is a polynomial algebra for $n=2,3$ for all $q$.  However, it is shown that for these excluded $q$ at least one of the generators $x_i$ in Theorem \ref{thm:intromaintheorem} is equal to $0$ in the skein algebra.  

It is well-known \cite{SikoraAlgGeomTop05} that the Sikora skein algebra is isomorphic to the CKM skein algebra for $n\leq3$.  As the present paper developed, it was proved by Poudel \cite{poudel2023comparisonslnspidercategories} that such an isomorphism also exists for $n>3$ for all $q$ except possibly for roots of unity of small order. Queffelec--Rose's proof of the above theorem was by categorical and representation theoretic methods; see also \cite{queffelecMR4715537}.  Moreover, their result is valid for any $q$, or more generally for the skein algebra viewed as a $\Z[q,q^{-1}]$-module.  Our proof of Theorem \ref{thm:intromaintheorem} has the advantage of being more connected to the underlying cluster geometry discussed later in this introduction.

The HOMFLYPT skein algebra $\mcal{S}(\surf)_\mathrm{HOMFLYPT}$ has a longer history than that of the skein algebra $\sk(\surf)$, especially in the case of the annulus $\surf=\ann$.  Turaev \cite{turaevMR0964255}, see also \cite{lickorishMR0918536}, proved that $\mcal{S}(\ann)_\mathrm{HOMFLYPT}$ is a polynomial algebra in infinitely many generators, corresponding to the non-trivial elements of the fundamental group $\pi_1(\ann)\cong\mbb{Z}$.  This algebra also connects to combinatorics and number theory, for example, the sub-algebra of $\mcal{S}(\ann)_\mathrm{HOMFLYPT}$ generated by $\Z_{>0}\subset\pi_1(\ann)$ can be thought of as an algebra of symmetric functions in infinitely many variables \cite{mortonMR1983094, mortonMR2439627}, and for $\surf$ the torus $\mcal{S}(\surf)_\mathrm{HOMFLYPT}$ is isomorphic to an elliptic Hall algebra \cite{mortonMR3626565} and naturally acts on $\mcal{S}(\ann)_\mathrm{HOMFLYPT}$ since $\surf$ is the boundary of a solid torus.  

On the other hand, as indicated above, the skein algebra $\sk(\surf)$ is more closely tied to the geometry and topology of character varieties.  More precisely, when $q=1$ then $\mcal{S}_n^1(\surf)$ is isomorphic to the unreduced coordinate ring of the $\nsl$ character variety of $\surf$ \cite{bullock1997rings, PrzytyckiBullPolishAcad91, SikoraTrans01}.  This means that the reduced coordinate ring is the quotient $\mcal{S}_n^1(\surf)/\sqrt{0}$ by the set of nilpotent elements.  The isomorphism sends a loop, or web, to the associated trace function on the character variety.  In particular, the proof of Theorem \ref{thm:intromaintheorem} gives another proof that the coordinate ring of the $\nsl$ character variety of the annulus is reduced.  See also \cite{FLORENTINO201432, sikora2014}.  

The skein algebra $\sk(\surf)$ is naturally a quotient of the $q$-evaluated HOMFLYPT skein algebra $\mcal{S}_n^q(\surf)_\mathrm{HOMFLYPT}$.  Said in another way, multiloops span $\sk(\surf)$ but are not linearly independent.  For instance, essentially by Theorem \ref{thm:intromaintheorem}, Turaev's loop basis for the HOMFLYPT skein algebra of the annulus becomes linearly dependent in $\sk(\ann)$; compare Lemma \ref{lem:inductionlargei} and Equation \eqref{eq:rootsinduction1}.  From the point of view of character varieties when $q=1$, this dependence reflects that there are non-trivial relations among trace functions of loops, best captured pictorially in terms of webs \cite{SikoraTrans01}.  Alternatively, when $q$ is not a root of unity of small order, then $\sk(\surf)$ is spanned by planar webs, namely webs without crossings; see Figure \ref{fig:more-relations}.  The basis elements for $\sk(\ann)$ of Theorem \ref{thm:intromaintheorem} are planar webs, thus providing a solution to the problem of finding a linearly independent spanning subset of planar webs.  

It is a non-trivial problem to find explicit small generating sets of skein algebras \cite{baseilhac2023unrestrictedquantummodulialgebras, Frohman22MathZ}, and harder still to find presentations.  Indeed, the latter has only been accomplished for a handful of surfaces; see \cite{cooke2022higherrankaskeywilsonalgebras, mortonMR3626565} for some more recent progress.  Theorem \ref{thm:intromaintheorem} provides a finite presentation for the skein algebra $\sk(\ann)$ of the annulus.  

Constructing explicit linear bases of skein algebras is another difficult problem.  For small values of $n$ bases have long been known \cite{hosteMR1238877, KuperbergCommMathPhys96}, namely the basis of simple essential multi-curves for $n=2$ and of non-elliptic webs for $n=3$, but bases have been more elusive for $n>3$ even for the simplest surfaces; see however \cite{fontaineMR2889146, gaetz2024rotationinvariantwebbaseshourglass, queffelecMR3729501}.  More recently, ideas from hyperbolic and cluster geometry have led to a general blueprint for constructing bases for $\sk(\surf)$, at least when the surface $\surf$ is punctured.  When $n$ is $2$ or $3$ this blueprint reproduces the previously mentioned bases, and the present paper can be considered as a demonstration of concept in the case $n=n$ when the surface is the annulus.

In more detail, such a blueprint is motivated by the concept from mirror symmetry of Fock--Goncharov duality  \cite{FockArxiv97, fockMR2233852, grossMR3758151} or, rather, Fock--Goncharov--Shen duality \cite{goncharovMR3418241, goncharov2022quantum} in the current setting.  Roughly speaking, this says that there is a canonical linear basis of the coordinate ring of the $\nsl$ character variety of $\surf $ whose elements are naturally indexed by a discrete sub-monoid $\fan_\idealtriang\subset\nnegZ^N$, the positive integer tropical points, defined with respect to an ideal triangulation $\idealtriang$ of the punctured surface $\surf$, and where $N$ is the dimension of the character variety.  Here, naturality means that this indexing behaves well under changing the triangulation, in particular transforming according to a tropical cluster mutation formula.  Such canonical bases have strong positivity properties; see for example Section \ref{sssec:quantum-traces-from-spectral-networks}, which relates to the physical point of view of Gaiotto--Moore--Neitzke \cite{gaiottoMR3250763}.  There is also a quantum version of this duality \cite{davisonMR4337972, fockMR2567745}.  From the skein theoretic perspective, being guided by the $n=2$ and $n=3$ cases, the most optimistic scenario when $n=n$ is that such bases are realized by planar webs that can be constructed by a local to global process, first constructing the webs locally in triangles, then gluing together these local webs across the triangulation to obtain the global basis web.  This is the procedure followed in the proof of Theorem \ref{thm:intromaintheorem}.

The key to a skein theoretic realization of Fock--Goncharov duality is the $\nsl$ quantum trace map \cite{bonahonMR2851072, ChekhovCzechJPhys00, douglasMR4717274, gabellaMR3613514, kim2022rm, le2023quantum, neitzkeMR4190271}.  This is an algebra homomorphism relating two quantizations of the character variety, namely the skein algebra and the Chekhov--Fock--Goncharov quantum higher Teichm\"{u}ller space.  More precisely, it is a map $\qtr$ from the skein algebra $\sk(\surf)$ to a quantum torus $\mcal{T}_n^q(\surf,\idealtriang)$, or a $q$-deformed Laurent polynomial algebra in $N$ variables, associated to the ideally triangulated punctured surface $\surf$.  Manifesting the tropical geometric nature of Fock--Goncharov duality, the highest term of the quantum trace polynomial of a basis web should have exponents given by the positive tropical integer points of that web.  This is true when $n$ is $2$ or $3$, and it is shown to be true when $n=n$ for the annulus in the course of proving Theorem \ref{thm:intromaintheorem}.  In particular, this is the fundamental property underlying the linear independence of the basis elements.  

As was somewhat unexpected to the authors, the quantum trace map also plays an important role in our proof in establishing the spanning property of the annulus basis. This is because the inadmissible values of $q$ in Theorem \ref{thm:intromaintheorem} include the, a priori, unknown roots of a family of polynomials.  These roots are forced to be roots of unity of small order by the quantum trace analysis; see Section \ref{ssec:main-result-of-the-section}.  

As a consequence of the proof of Theorem \ref{thm:intromaintheorem}, we establish the following parallel applications relating quantum algebras associated to the annulus $\ann$.

\begin{introtheorem}\label{thm:introtheorem2}
The quantum trace map for the skein algebra of the annulus is injective for every admissible value of $q$ as in Theorem {\upshape\ref{thm:intromaintheorem}}.  In particular, the quantum trace map furnishes an algebra embedding of the polynomial algebra $\C[x_1,x_2,\dots,x_{n-1}]$ into a quantum torus algebra $\mcal{T}_n^q(\ann,\idealtriang)$ depending on the choice of an ideal triangulation $\idealtriang$ of the annulus.  
\end{introtheorem}

See Corollary \ref{cor:firstapplication} in Section \ref{ssec:injectivity}.  Note that the injectivity of the quantum trace map is known for general punctured surfaces $\surf$ when $n\leq3$ for all $q$.  

\begin{introtheorem}\label{thm:introtheorem3}
The splitting map \cite{leMR3827810, LeArxiv21} for the skein algebra of the annulus is injective for every admissible value of $q$ as in Theorem {\upshape\ref{thm:intromaintheorem}}.  In particular, the splitting map furnishes a canonical algebra embedding of the polynomial algebra $\C[x_1,x_2,\dots,x_{n-1}]$ into the stated skein algebra $\mcal{S}^q_n(\ann)_\mathrm{st}$.
\end{introtheorem}

See Corollary \ref{cor:secondapplication} in Section \ref{ssec:injectivityofthesplittingmap}.  Note that the injectivity of the splitting map is known for general punctured surfaces $\surf$ when $n\leq3$ for all $q$  \cite{higginsMR4609753, leMR3827810} and when $n=n$ for $q=1$ \cite{korinman2023triangulardecompositioncharactervarieties, wang2023statedslnskeinmodules} and for certain roots of unity $q$ of small order \cite{wang2024statedslnskeinmodulesroots}. 

For some other more recent work on the skein theory of the annulus, see \cite{beaumontgould2023powersumelementsg2, bonahon2023centralelementsmathrmsldskeinalgebra, douglasMR4688540, queffelecMR4715537}.  

\section*{Acknowledgements}

We thank many people for all of the enjoyable and helpful conversations as this project developed, including but not limited to: Francis Bonahon, Noah Charles, Charlie Frohman, Vijay Higgins, Rick Kenyon, Hyun Kyu Kim, Thang L\^{e}, Andy Neitzke, Nick Ovenhouse, Sam Panitch, Haolin Shi, Adam Sikora, Zhe Sun, Sri Tata, and Tao Yu.  In particular, we are grateful to Zhe Sun for teaching us about $\nsl$ tropical web coordinates during the height of the pandemic in January 2021.  Much of this work was completed during memorable visits at Yale University, the Max Planck Institute for Mathematics in the Sciences, and Trinity College Dublin to whom we are very thankful for their warm hospitality.  Lastly, we thank the referee for suggestions for improvements.

\section*{Preliminary notations and conventions}

Let $n$ be an integer $\geq1$.  Denote by $\sC$ the set of non-zero complex numbers, $\sC=\C\setminus\{0\}$.  Let $\qroot\in\sC$ and put $q=(\qroot)^n$.  Starting in Section \ref{sssec:quantum-trace-map}, we will also fix an arbitrary $(2n)$-root $\wroot^{1/2}$ of $\qroot$.  So $(\wroot^{1/2})^{2n}=\qroot$.  

The natural numbers are denoted by $\nnegZ=\Z_{\geq0}$. For a natural number $m \geq 1$, define the quantum natural number $[m]=\sum_{i=1}^m q^{-m-1+2i}=(q^m-q^{-m})/(q-\inv{q})$, and define the quantum factorial $[m]!=\prod_{i=1}^m [i]$.  Put also $[0]=0$ and $[0]!=1$.  Note that $[m]=0$ if and only if $q\in\sC$ is a $(2m)$-root of unity and $q\neq\pm 1$.  

Vector spaces $V$ will be over $\C$.  The notation $v\propto w$ means $v=a w$ for some scalar $a$.    If $A$ is a, possibly non-commutative, algebra with $1$ over $\C$, denote by $\sub{a_1,a_2,\dots,a_m}$ the sub-algebra generated by elements $a_1,a_2,\dots,a_m\in A$, namely the linear span over $\C$ of all monomials of the form $a_{i_1}a_{i_2}\cdots a_{i_p}$ for $i_j\in\{1,2,\dots,m\}$, where this monomial is $1$ if $p=0$.  

Let $\sym_m$ denote the symmetric group on $m$ elements, so $\sym_m$ is the set of self bijections of $\{1,2,\dots,m\}$.   We denote permutations $\sigma\in\sym_m$ by $\bmat\sigma(1)&\sigma(2)&\hdots&\sigma(m)\emat$.  For example, $\bmat3&1&2&4\emat\in \sym_4$ denotes the permutation where $1\mapsto3$, $2\mapsto1$, $3\mapsto2$, $4\mapsto4$.  The order of multiplication in $\sym_m$ coincides with the usual convention for multiplying functions, so if $\sigma_1,\sigma_2\in\sym_m$, then the composition $\sigma_2\circ\sigma_1\in\sym_m$ is defined by $(\sigma_2\circ\sigma_1)(i)=\sigma_2(\sigma_1(i))$ for $i=1,2,\dots,m$.    The \define{length} $\ell(\sigma)$ of a permutation is the minimum number of factors appearing in a decomposition of $\sigma$ as the product of adjacent transpositions.  

Curves in two- and three-dimensional spaces are always oriented.  When permutations $\sigma\in\sym_m$ are represented by overlapping upward oriented curves in a box, they are read from bottom to top.  For example, the picture on the left hand side of Figure \ref{fig:permutation-convention} represents the permutation $\bmat3&1&2&4\emat\in \sym_4$.  

\begin{figure}[htb!]
\centering
\includegraphics[scale=0.8]{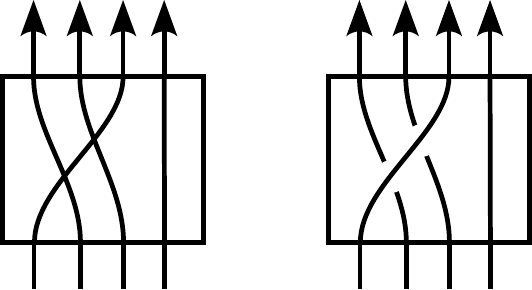}
\caption{The permutation $\bmat3&1&2&4\emat$ and its corresponding positive permutation braid.}
\label{fig:permutation-convention}
\end{figure}

These multi-curves are considered up to homotopy fixing endpoints on the box boundary.  In particular, the composition $\sigma_2\circ\sigma_1$ corresponds, up to homotopy, to placing the picture for $\sigma_2$ above the picture for $\sigma_1$.  The length $\ell(\sigma)$ becomes the number of crossings in the picture for $\sigma$ represented by a multi-curve in minimal position.  

We also think of elements $\beta$ of the braid group $\braid_m$ pictured in such a box, where the braids are likewise oriented from bottom to top.  For example, on the right hand side of Figure \ref{fig:permutation-convention} is the \define{positive lift} braid $\lift{\sigma}_+$ projecting to $\sigma$.  Here, a crossing is positive if a counterclockwise rotation goes from the outgoing over strand to the outgoing under strand, and is negative otherwise.  Define similarly the \define{negative lift} braid $\lift{\sigma}_-$.  Note that $\lift{\sigma}_+$ and $\lift{\sigma}_-$ are inverses in the braid group.  Braids are considered up to isotopy fixing endpoints on the box boundary.  The composition $\beta_2\circ\beta_1\in\braid_m$ of braids is the result of placing the picture for $\beta_2$ above that for $\beta_1$.  If a braid $\beta^\prime$ is the same as $\beta$ up to crossing changes, we call $\beta^\prime$ a  \define{mutant} of $\beta$.

We use the notation $\proj{\beta}\in\sym_m$ to denote the permutation projected to by the braid $\beta$.  The multi-curve $\proj{\beta}$ has no monogons, but, before any homotopy, could have bigons.  Here, by \define{monogon} and \define{bigon}, we mean as in Figure \ref{fig:reidemeister}, where on the left hand side the contractible enclosed region is a monogon, and where on the right hand side the contractible enclosed region is a bigon.  Note that it is possible for this enclosed region to intersect other parts of the multi-curve.  As an example of this subtlety about homotopy, the braid $\beta\in\braid_2$ representing a full twist satisfies that $\proj{\beta}=\bmat2&1\emat^2$ has a bigon, even though as a permutation it is trivial after homotopy.  The permutation picture $\proj{\beta}$ can always be homotoped into minimal position, that is, so that it has no bigons.  We call $\beta$ a \define{permutation braid} if $\beta$ can be isotoped such that $\proj{\beta}$ has no bigons.  Every permutation $\sigma$ has $2^{\ell(\sigma)}$ \define{lifts}, denoted by $\lift{\sigma}$, which are just the permutation braids projecting to $\sigma$.    For example, $\lift{\sigma}_+$ and $\lift{\sigma}_-$ are two such lifts.  Note that the permutation braids lifting the same permutation are just the mutants of $\lift{\sigma}_+$.  

\begin{figure}[htb!]
\centering
\includegraphics[scale=0.8]{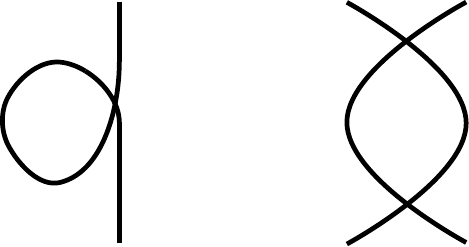}
\caption{Monogon and bigon.}
\label{fig:reidemeister}
\end{figure}

\section{Setting}\label{sec:setting}

\subsection{Webs}\label{ssec:webs}

Consider an edge oriented graph with finitely many vertices, possibly with multiple edges between vertices and possibly with more than one component, that is properly embedded in $\threereal$ such that, in addition:
\begin{enumerate}
\item  the graph is $n$-valent, namely, there are $n$ edges incident to each vertex;
\item  each vertex is either a sink or source, namely, the orientations of the edges at the vertex go all in or all out, respectively;
\item  the graph can contain loop components, namely, components without any vertices and which are homeomorphic to a circle;
\item  the ends of the incident edges at each vertex are distinct and lie in a two-dimensional plane in $\threereal$;
\item  the graph is equipped with a \define{framing}, meaning a continuous assignment of a unit normal vector at each point of the graph, in the following sense:  at a nonvertex point, lying on some unique edge or loop component, this unit vector is in the plane perpendicular to the tangent of the edge at the point; and, at a vertex, this unit vector is perpendicular to the two-dimensional plane spanned by the incident edges;
\item  this framing determines a \define{cyclic order} of the \define{half-edges} incident at each vertex, according to the right hand rule for sources and the left hand rule for sinks; moreover, at each vertex we choose, as part of the data, a \define{linear order} of these half-edges compatible with the cyclic order, which is equivalent to choosing a distinguished half-edge.
\end{enumerate}
Such data is called a \define{based $\nsl$ web} (or just \define{web}) \define{in $\threereal$}, denoted by $W$.   If one forgets the linear order at each vertex, just remembering the cyclic order, we call this an \define{unbased web in $\threereal$}, denoted by $\unbased{W}$.  Webs are considered up to \define{isotopy} through the family of webs.  The \define{empty web} and \define{empty unbased web}, denoted by $\emptyset$, are the unique representatives of their respective isotopy classes.  

Alternatively, one may think of a web as an oriented $n$-valent ribbon graph, obtained by gluing rectangle edges to disk vertices, and equipped with a distinguished half-edge at each vertex.  Note here that there are two orientations being considered, the edge orientations and the orientation as a surface.  In this way of thinking, the framing vector, normal to the tangent plane of the ribbon at each point, is determined by the surface orientation.

In figures, we represent webs by \define{diagrams} in the plane $\tworeal\cong\tworeal\times\{0\}\subset\threereal$, which includes over and under crossing data.  A diagram determines a web, up to isotopy, by equipping the diagram with the \define{blackboard framing}, namely the constant unit framing in the positive vertical direction, that is, the positive third coordinate of $\threereal$ thought of as pointing out of the page toward the eye of the reader.  Note that, by possibly introducing kinks, any web can be isotoped to have the blackboard framing, and so can be represented by a diagram.  We indicate in diagrams the choice of the preferred  half-edge at each vertex, determining the linear order, by placing a \define{cilia} just before the preferred edge in the cyclic order, again, counterclockwise at sources and clockwise at sinks.  See, for example, Figure \ref{fig:sikora-relations}(D) where the cilia is colored red.

\begin{figure}[htb!]
\centering
\includegraphics[width=.8\textwidth]{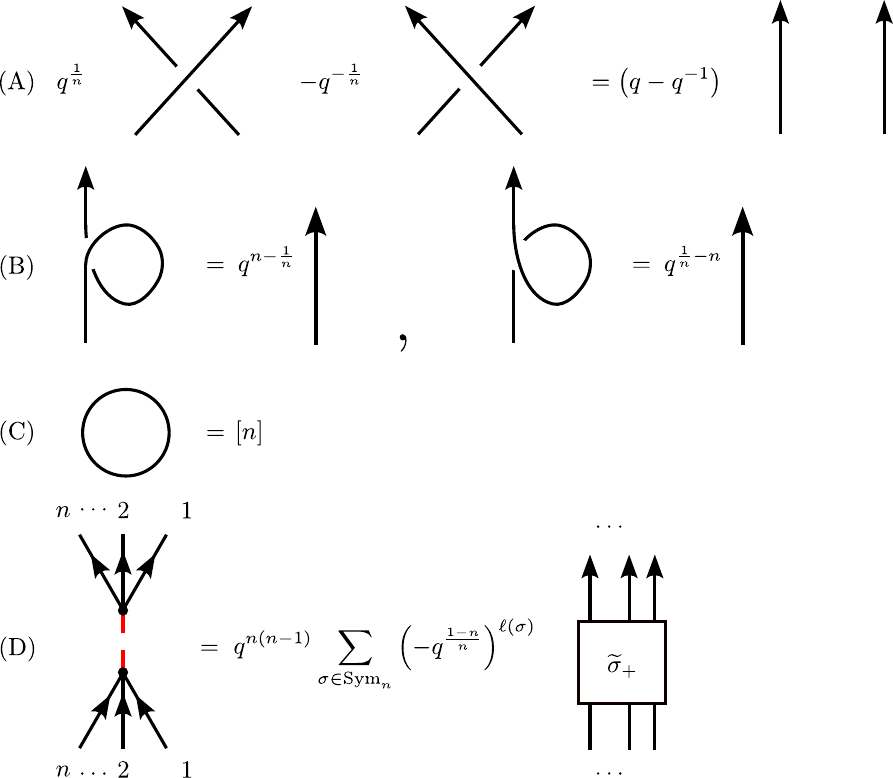}
\caption{Defining relations for the based skein algebra.}
\label{fig:sikora-relations}
\end{figure}

Let $\ann$ denote the \define{open annulus} (or just \define{annulus}) $\ann=\{(x,y,0)\in\threereal;0<x^2+y^2<\infty\}$.  Let $\intt=(0,1)\subset\R$ be the \define{open unit interval} (or just \define{interval}).  The \define{thickened annulus} is $\ann\times\intt\subset\threereal$.  A \define{web in the thickened annulus} is a web in $\threereal$ that is contained in $\ann\times\intt$, considered up to isotopy staying within the thickened annulus.  An \define{unbased web in the thickened annulus} is defined similarly.  From here on, `web' will always mean `web in the thickened annulus', unless stated otherwise.  
       
\begin{remark}\label{rem:websfor3manifolds}
Essentially the same definition allows one to define based and unbased webs  more generally in three-dimensional manifolds $\man$ (in place of $\ann\times\intt$).  See, for example, \cite{LeArxiv21}.
\end{remark}  

By a \define{planar isotopy} of a web, we mean an isotopy that can be achieved at the level of diagrams on the surface by an ambient isotopy of the surface.  In particular, planar isotopies preserve the number of crossings of the diagram.  

\subsection{Skeins}\label{ssec:skeins}
	
Equip the thickened annulus $\ann\times\intt\subset\threereal$ with the orientation inherited as an open sub-set of $\threereal$, the latter having its usual orientation.

The \define{based $\nsl$ skein module} (or just \define{skein module}) $\sk=\sk(\ann\times\intt)$ is the vector space obtained by quotienting the vector space of formal finite linear combinations of isotopy classes of webs $W$ in $\ann\times\intt$ by the relations shown in Figure \ref{fig:sikora-relations}.  Here, the webs appearing in a relation differ in a ball, as shown, and agree outside of this ball.  We emphasize that the relations concern the isotopy classes of the representative webs displayed in the pictures of the relations.  Note that we are abusing notation, since the skein module $\sk$ actually depends on $\qroot$ and not just $q$.  

We refer to the relations (A), (B), (C), (D) in Figure \ref{fig:sikora-relations} as the \define{crossing change}, \define{kink removing}, \define{unknot removing}, and \define{vertex removing} \define{relations}, respectively.  Note (C) is independent of orientation.

The \define{unbased skein module} $\usk$ is the vector space obtained by quotienting the vector space of formal linear combinations of isotopy classes of unbased webs $\unbased{W}$ by the relations shown in Figure \ref{fig:sikora-le-relations}.  
For clarity, we sometimes use the notation $\class{W}$ to indicate the class of a web $W$ in a skein module, not to be confused with the notation $[m]$ for quantum integers.

\begin{figure}[htb!]
\centering
\includegraphics[width=.8\textwidth]{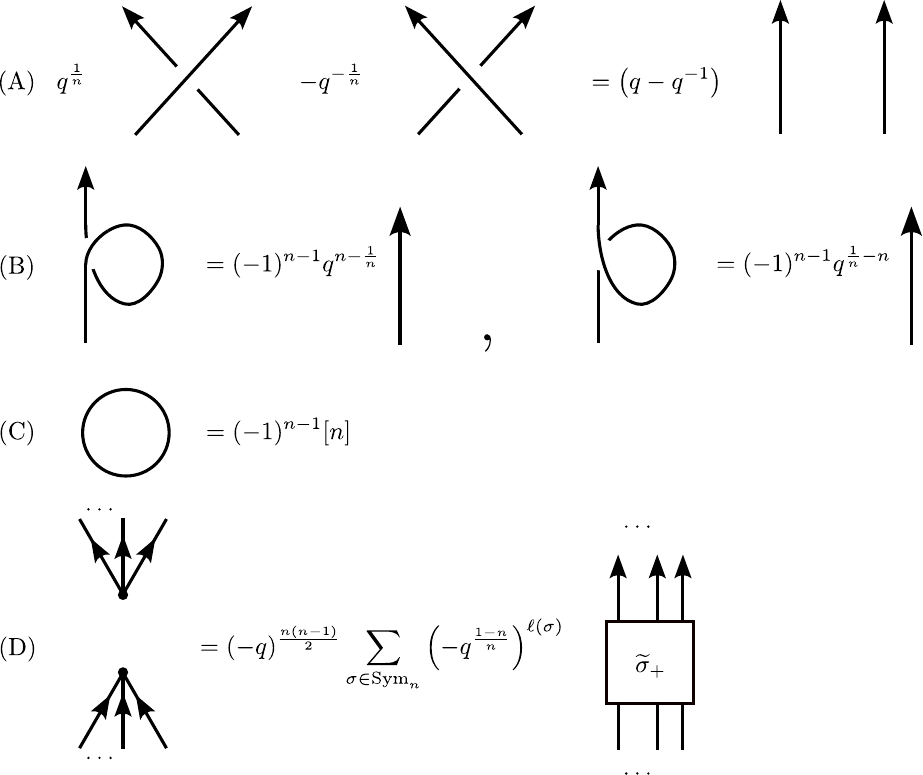}
\caption{Defining relations for the unbased skein algebra.}
\label{fig:sikora-le-relations}
\end{figure}  

The skein modules $\sk$ and $\usk$ have an algebra structure, the multiplication being defined by stacking diagrams, where multiplying left to right means stacking from top to bottom.  In other words, the product $WW^\prime$ is obtained by rescaling $W\subset\ann\times\intt$ into $W\subset\ann\times(1/2,1)$ and rescaling $W^\prime\subset\ann\times\intt$ into $W^\prime\subset\ann\times(0,1/2)$ followed by inserting both of these rescaled webs into $\ann\times\intt$.  Similarly for unbased webs.   We then refer to the \define{skein algebra} $\sk$ and the \define{unbased skein algebra} $\usk$.  Because of the topology of the annulus, we see that $\sk$ and $\usk$ are, in fact, commutative algebras, so  $WW^\prime=W^\prime W$.  The unit in these algebras is the class $\class{\emptyset}$ of the empty web or empty unbased web.  

\begin{proposition}[\cite{LeArxiv21}]\label{prop:changeciliachangesign}
Assume $W$ and $W^\prime$ are webs such that $W$ and $W^\prime$ agree except at a single vertex where the distinguished half-edge of $W^\prime$ is just after, or before, that of $W$ with respect to the cyclic order.  Then $W^\prime=(-1)^{n+1}W$ in $\sk$.  \qed
\end{proposition}

We say such a $W$ and $W^\prime$ differ by a \define{cilia change relation}.  Consequently, for all intents and purposes, one can forget about the cilia data when $n$ is odd.

In Figure \ref{fig:more-relations} we provide some more useful relations.  We refer to the relations (A) and (B) in the figure as the \define{planarizing} and \define{peeling relations}, respectively.  Note that the planarizing relation cannot be used when $[n-2]!=0$, equivalently, when $q$ is a $(2m)$-root of unity for some $m=2,3,\dots,n-2$ and $q\neq\pm1$. 

\begin{figure}[htb!]
\centering
\includegraphics[width=\textwidth]{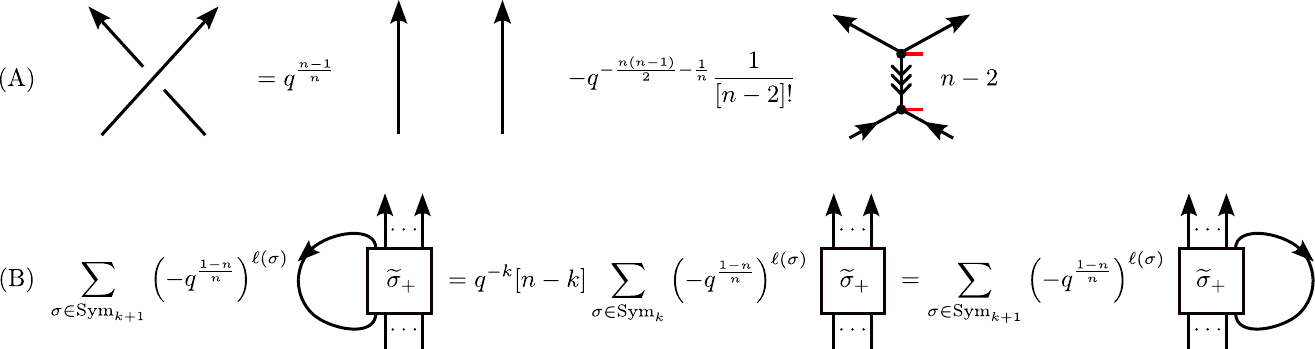}
\caption{More relations for the based skein algebra.}
\label{fig:more-relations}
\end{figure}

The skein algebra $\sk$ is denoted in the notation of \cite{LeArxiv21} by $\mcal{S}_n^b(M;u)$ for $M=\ann\times\intt$ and $u=(-q)^{n(n-1)/2}$, and was first defined, for general $n$, in \cite{SikoraAlgGeomTop05}.  The unbased skein algebra $\usk$ was first defined, for general $n$, in \cite{LeArxiv21} and is denoted by $\mcal{S}_n(M,\emptyset)$ in \cite{LeArxiv21}.  

It is explained, in \cite{LeArxiv21}, how, for $n$ even, the choice of a spin structure $s$ on $\ann\times\intt$ determines an assignment $W\mapsto s(W)\in\{\pm 1\}$ for each web $W$, defined independently of isotopy.  Recall that $\unbased{W}$ is the underlying unbased web.

\begin{proposition}[\cite{LeArxiv21}]\label{prop:spinstructures}
Let $u=(-q)^{n(n-1)/2}$.
\begin{enumerate}
\item  For $n$ odd, the assignment $W\mapsto u^{\#\mrm{sinks}(W)}\unbased{W}$ determines an isomorphism of algebras $\sk\overset{\sim}{\to}\usk$.
\item  For $n$ even, and for a choice of spin structure $s$ on $\ann\times\intt$, the assignment $W\mapsto s(W)u^{\#\mrm{sinks}(W)}\unbased{W}$ determines an isomorphism of algebras $\sk\overset{\sim}{\to}\usk$, with inverse sending $\unbased{W}$ to $s(W)(\inv{u})^{\#\mrm{sinks}(W)}W$ for any  web $W$ such that $\unbased{W}$ underlies $W$.  \qed
\end{enumerate}
\end{proposition}

In this paper, we prefer to work with the based skein algebra $\sk$.  However, by Proposition \ref{prop:spinstructures}, we immediately obtain:

\begin{corollary}\label{cor:basesarethesame}
For all $\qroot\in\sC$, a collection $\{W_i\}_{i\in I}$ of  webs forms a linear basis for $\sk$ if and only if the  collection of underlying unbased webs $\{\unbased{W_i}\}_{i\in I}$ forms a basis for $\usk$.  \qed
\end{corollary}

\begin{remark}\label{rem:surfaces}
Essentially the same definition allows one to define based and unbased skein modules more generally for oriented three-dimensional manifolds $\man$ (Remark \ref{rem:websfor3manifolds}), and to define based and unbased skein algebras for thickened  oriented surfaces $\surf\times\intt$.  Note that the skein algebra for a surface $\surf$ is non-commutative in general.  
\end{remark}

\section{Main result}\label{sec:mainresult}

For $i=1,2,\dots,n-1$, define the \define{$i$-th irreducible basis web} $B_i\in\sk$ as in Figure \ref{fig:basiselement}.  Here, we have used a pictorial shorthand, where $i$ (resp. $n-i$) indicates $i$ (resp. $n-i$) parallel edges, oriented as shown.  The cilia (Section \ref{ssec:webs}) are colored red.
The cross indicates the puncture not at infinity.  

\begin{figure}[htb!]
\centering
\includegraphics[scale=.9]{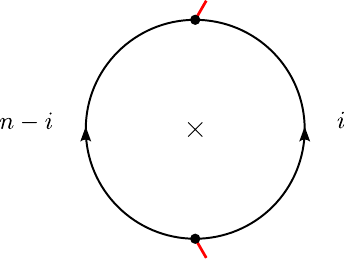}
\caption{The $i$-th irreducible basis web.}
\label{fig:basiselement}
\end{figure}

The following statement is the main result of the paper.  
   
\begin{theorem}\label{thm:maintheorem}
Let $\badset\subset\sC$ be the set of $\qroot$ such that $q=(\qroot)^n$ is a $(2m)$-root of unity for some $m=2,3,\dots,n-1$ and $q\neq\pm1$.   Then for all $\qroot \in \sC\setminus\badset$, the  skein algebra $\sk$ is a commutative polynomial algebra  in $n-1$ variables, with generators the $n-1$ irreducible basis webs $B_1,B_2,\dots,B_{n-1}$.  If $\qroot\in\badset$, then $B_i=0\in\sk$ for some $i=1,2,\dots,n-1$.  
\end{theorem}

The proof of the last sentence of the theorem is not hard, and is given in Section \ref{proof-of-the-last-statement-of-maintheorem}.  The theorem is proved in Section \ref{sssec:annuluscomputationfinishing}.  

Since the skein algebra $\sk$ is commutative, there is an algebra map $\C[x_1,x_2,\dots,x_{n-1}]\to\sk$ from the commutative polynomial algebra in $n-1$ variables, defined by sending $x_i$ to $B_i$.  To prove the theorem, we need to show that this map is injective and surjective, for $\qroot\in\sC\setminus\badset$.  
In other words, the monomials of the form $B_{1}^{k_1}B_{2}^{k_2}\cdots B_{n-1}^{k_{n-1}}\in\sk$ for $k_i\in\nnegZ$ are linearly independent, and these monomials span, that is, $\sk=\sub{B_1,B_2,\dots,B_{n-1}}$.  

By Corollary \ref{cor:basesarethesame}, we immediately obtain:

\begin{corollary}\label{cor:mainthmforunbased}
 Theorem {\upshape\ref{thm:maintheorem}} is also true for the unbased skein algebra $\usk$.  \qed
 \end{corollary}
 
Combining Theorem \ref{thm:maintheorem} and Corollary \ref{cor:mainthmforunbased}, we gather:
 
\begin{corollary}\label{cor:finitegenerated}
For $\qroot\in\sC\setminus\badset$, the skein algebras $\sk$ and $\usk$ are finitely generated, with an explicit finite presentation.  Indeed, they are isomorphic to the commutative polynomial algebra $\C[x_1,x_2,\dots,x_{n-1}]$. \qed
\end{corollary}

\section{Spanning}\label{sec:spanning}

In this section, we prove the last sentence and the spanning property of Theorem \ref{thm:maintheorem}.  In Section \ref{ssec:spanningsetup}, we set some notation.  In Section \ref{proof-of-the-last-statement-of-maintheorem}, we prove the last sentence of Theorem \ref{thm:maintheorem}.  In Section \ref{ssec:intermediate-statement}, we prove Proposition \ref{prop:gammas-generate}, which says that an auxiliary infinite set of links generates the skein algebra $\sk$.     In Section \ref{ssec:main-statement-of-the-section}, we state the almost-main result of the section, Proposition \ref{prop:gammastobasis}, namely the spanning property of the theorem, but for a superset $\altbadset\supset\badset$.  In Section \ref{ssec:induction-step}, we prove the inductive part of Proposition \ref{prop:gammastobasis}.  In Section \ref{ssec:small-skeins}, we finish the proof of Proposition \ref{prop:gammastobasis}.  In Section \ref{ssec:main-result-of-the-section}, we prove $\altbadset=\badset$, which uses the main result of Section \ref{sec:linear-independence}, establishing the main result of this section.  In Section \ref{ssec:a-result-for-general-surfaces}, we digress to discuss a generalization of Proposition \ref{prop:gammas-generate} valid for general surfaces.  

\subsection{Setup}\label{ssec:spanningsetup}

For a permutation $\sigma\in\sym_m$, the \define{permutation multi-curve} associated to $\sigma$, which by abuse of notation we also denote by $\sigma$, is the oriented multi-curve in $\ann$ obtained from the permutation picture for $\sigma$ in a box in $\ann$ by connecting the boundary strands with parallel curves going counterclockwise around the annulus.  See for example the leftmost picture in Figure \ref{fig:k-generator} which is the permutation multi-curve $\sigma=\bmat3&1&2&4\emat$.  Permutation multi-curves are considered up to homotopy in the annulus.  Note that the composition $\sigma_2\circ\sigma_1$ of permutations makes sense as a permutation multi-curve.  

\begin{figure}[htb!]
\centering
\includegraphics[width=0.8\textwidth]{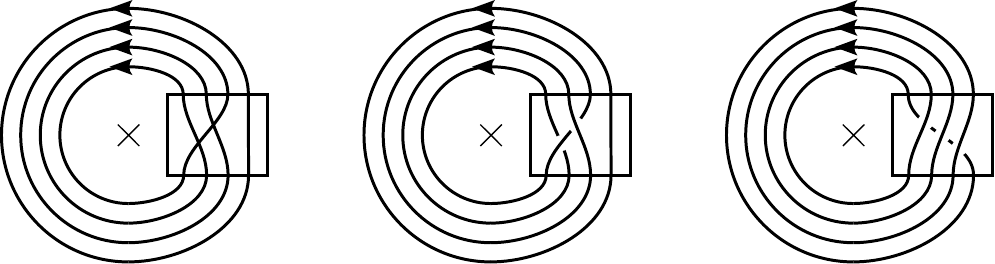}
\caption{Permutation multi-curve, braid link, and permutation knot.  Note that the third permutation is different from the first two.}
\label{fig:k-generator}
\end{figure}

Similarly, for a braid $\beta\in\braid_m$ we can consider the \define{braid link} associated to $\beta$, which by abuse of notation we also denote by $\beta$, as the link in $\ann\times\intt$ obtained by closing up the boxed braid counterclockwise around the thickened annulus.  See for example the middle and rightmost pictures in Figure \ref{fig:k-generator}.  Note that the rightmost braid is the positive braid $\lift{\sigma}_+$ for $\sigma=\bmat2&3&4&1\emat$.  Braid links are considered up to isotopy in the thickened annulus.  Note that the composition $\beta_2\circ\beta_1$ of braids makes sense as a braid link.  We will also think of a braid link $\beta$ as an element of $\sk$.  Note, in particular, that the composition $\beta_2\circ\beta_1$ is different from the product $\beta_2\beta_1$ in the skein algebra $\sk$.  The \define{reverse braid link} $\rev{\beta}$ is obtained by reversing the orientation of the braid link $\beta$.  We refer to a connected braid link $\beta$ as a \define{braid knot}.  A \define{permutation link} is a braid link associated to a permutation braid $\beta=\lift{\sigma}$ for some $\sigma\in\sym_m$.  A \define{permutation knot} is a connected permutation link.  

For any link $L$ in $\ann\times\intt$, we use the notation $\proj{L}$ to indicate the multi-curve projection of the link on the annulus.  This generalizes the notation $\proj{\beta}\in\sym_m$ for a braid link $\beta\in\braid_m$.  

\subsection{Proof of the last sentence of Theorem \ref{thm:maintheorem}}\label{proof-of-the-last-statement-of-maintheorem}

Given the $i$-th irreducible basis web $B_i$, we can isotopy it into either the \define{positive} or \define{negative position}, as illustrated in Figure \ref{fig:fig8}, where the positive (resp. negative) position is displayed on the left (resp. right) hand side of Figure \ref{fig:fig8}. 

\begin{figure}[htb!]
\centering
\includegraphics[width=0.8\textwidth]{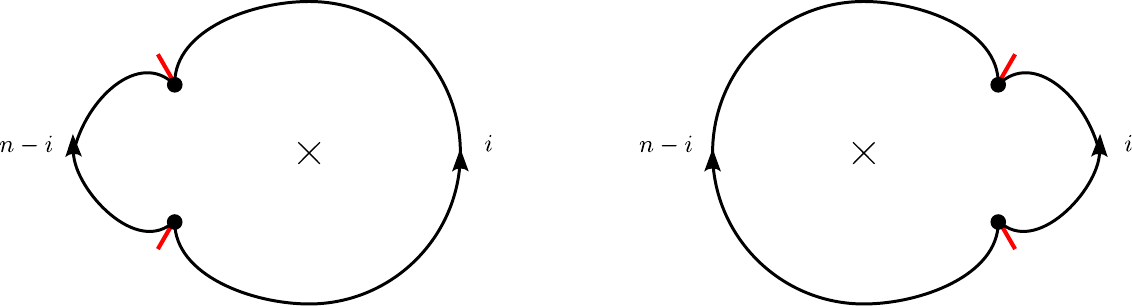}
\caption{Positive and negative position of the $i$-th irreducible basis web.}
\label{fig:fig8}
\end{figure}

Starting in the positive position, we can write 
\begin{equation*}
B_i=q^{n(n-1)}\lp\prod_{k=i}^{n-1}q^{-k}[n-k]\rp\sum_{\sigma\in\sym_i}\lp-q^{(1-n)/n}\rp^{\ell(\sigma)}\beta_{i,\sigma}\in\sk
\end{equation*}
by first applying the vertex removing relation, followed by the peeling relation to kill the $n-i$ strands.  Here the $\beta_{i,\sigma}=\lift{\sigma}_+\in\braid_i$ are the braid links created from the vertex removing relation.  Similarly, starting in the negative position, we can write
\begin{equation*}
B_{n-i}=q^{n(n-1)}\lp\prod_{k=i}^{n-1}q^{-k}[n-k]\rp\sum_{\sigma\in\sym_i}\lp-q^{(1-n)/n}\rp^{\ell(\sigma)}\rev{\beta_{i,\sigma}}\in\sk.
\end{equation*}

\begin{proposition}\label{lem:lastsentenceofmaintheorem}
Let $\badset\subset\sC$ be as in Theorem {\upshape\ref{thm:maintheorem}}.  If $\qroot\in\badset$, then $B_i=0\in\sk$ for some $i=1,2,\dots,n-1$. 
\end{proposition}

\begin{proof}
Note that $\badset=\emptyset$ for $n=2$.  So assume $n\geq3$.  From the two equations above, we see that $B_i=0\in\sk$ and $B_{n-i}=0\in\sk$ when $[n-k]=0$ for some $k=i,i+1,\dots,n-1$.  If $\qroot\in\badset$, so $q$ is a $(2l)$-root of unity for some $l=2,3,\dots,n-1$ and $q\neq\pm1$, then $[l]=0$.  So $B_i=0\in\sk$ and $B_{n-i}=0\in\sk$ for all $1\leq i\leq n-l$.\end{proof}

\subsection{Countable generating set}\label{ssec:intermediate-statement}

For $m\in\Z$, define the \define{$m$-th power knot} $\gamma_m$ in $\ann\times\intt$ as follows.  For $m>0$, we put $\gamma_m$ to be the permutation knot $\gamma_m=\lift{\bmat2&3&\hdots&m&1\emat}_+$.  For $m<0$, $\gamma_m$ is the reverse permutation knot  $\rev{\gamma_{|m|}}$. And $\gamma_0=\emptyset$.  For example, the rightmost picture in Figure \ref{fig:k-generator} is $\gamma_4$.  We will also think of $\gamma_m$ as an element of $\sk$.  Note that in the fundamental group of the annulus, the projection $\proj{\gamma_m}$ represents the $m$-th power of the counterclockwise generator.  Then, we have:

\begin{proposition}\label{prop:gammas-generate}
For any $\qroot\in\sC$, we have $\sk=\sub{\gamma_m;m\in\Z}$.  \end{proposition}

The primary content of Proposition \ref{prop:gammas-generate}, which is that the curves $\gamma_m$ generate the \define{HOMFLYPT skein algebra}, namely framed oriented links considered up to the crossing change, kink removing, and unknot removing relations, is well known \cite{lickorishMR0918536}.  For completeness, and because we will need the essential statement (Lemma \ref{lem:isotopic-in-the-annulus}) to prove some more refined statements later on (Lemmas \ref{lem:refinement} and \ref{lem:leading-skein-term}), we provide a detailed proof of Proposition \ref{prop:gammas-generate}.  

\begin{lemma}\label{lem:isotopic-in-the-annulus}
For $m>0$, let $\beta=\lift{\sigma}\in\braid_m$ be a permutation knot.  Then $\gamma_m$ is isotopic in the annulus to a mutant of $\beta$.  Similarly, the reverse $\rev{\gamma_m}=\gamma_{-m}$ is isotopic in the annulus to the reverse  of a mutant of $\beta$.
\end{lemma}

\begin{proof}
The algorithm of \cite{lickorishMR0918536} to determine which crossings need to be changed, namely which mutant of $\beta$ to take, is to change the crossings in such a way as to obtain the descending knot, namely the knot which  decreases in height until  climbing up at the very end.  Once the knot is  descending, it is clear that it is isotopic in the annulus to $\gamma_m$.  
\end{proof}

\begin{remark}
A slight upgrade of Lemma \ref{lem:isotopic-in-the-annulus}, which we will not need, is that if moreover $\beta$ has $m-1$ positive crossings and no negative crossings, then (resp. the reverse of) $\beta$ is isotopic in the annulus to (resp. the reverse of) $\gamma_m$.    
\end{remark}

For links $L$ and $L^\prime$ and a natural number $N$, we will use the notation $\class{L}\indeq\class{L^\prime}$ to indicate that $\class{L}=a^\prime\class{L^\prime}+\sum_i a_i\class{L_i}\in\sk$ for scalars $a^\prime\neq0,a_i$ such that the links $L_i$ have no more than $N-1$ crossings.    We say $L$ and $L^\prime$ are \define{equivalent up to $N$ crossings}.  For example, if $L$ and $L^\prime$ are links appearing as on the left hand side of the crossing change relation in Figure \ref{fig:sikora-relations}, then $\class{L}\indeq\class{L^\prime}$ for $N$ the number of crossings of $L$.  In particular, if $\beta\in\braid_m$ is a braid link, and $\beta^\prime\in\braid_m$ is a mutant of $\beta$, then $\class{\beta}\indeq\class{\beta^\prime}$ and $\class{\lp\rev{\beta}\rp}\indeq\class{\lp\rev{\beta^\prime}\rp}$.  

\begin{lemma}\label{lem:annulus-lemma}
If $K$ is any non-contractible knot, not necessarily a braid knot, in the thickened annulus with $N$ crossings, then there is a permutation knot $\beta=\lift{\sigma}\in\braid_m$ for some $m>0$ such that $\class{K}\indeq\class{\beta}$ or $\class{K}\indeq\class{\lp\rev{\beta}\rp}$.  
\end{lemma}

\begin{proof}
By a planar isotopy (Section \ref{ssec:webs}), take the knot into a box such that outside the box the knot consists of parallel strands going around the annulus.  Note that the knot intersects only the top and bottom edges of the box.  There might be U-turns on either edge of the box.  An innermost U-turn is a U-turn that does not completely contain another U-turn.  Given any U-turn, by iteratively applying the crossing change relation, there is a knot $K_1$ such that $\class{K}\indeq\class{K_1}$ and in $K_1$ the U-turn lies above all other strands.  If this U-turn is innermost, then there is a planar isotopy of $K_1$ removing the U-turn from the box and sliding it around the annulus into the other side of the box.  Repeat this process, yielding a knot $K_p$ with no U-turns and $\class{K}\indeq\class{K_p}$.  Note that $K_p$ has $m$ strands, for some $m>0$, with the same orientation, all either going from bottom to top or from top to bottom.  If $K_p$ has a bigon, which we can take to be innermost, by iteratively applying the crossing change relation there is a knot $K_{p+1}$ such that $\class{K_p}\indeq\class{K_{p+1}}$ and in $K_{p+1}$ an edge of the bigon lies above all other strands.  By an isotopy, but not a planar isotopy, of $K_{p+1}$, the bigon can be removed, after which the number of crossings of $K_{p+1}$ has strictly decreased.  In a similar way, a monogon can be removed by applying crossing change relations followed by the kink removing relation.  By repeating this process, we obtain a knot $K_{q}$ with no U-turns, no bigons, no monogons, and $\class{K}\indeq\class{K_{q}}$.  If $K_q$ is oriented upward, then $K_q=\beta=\lift{\sigma}\in\braid_m$ is a permutation knot, else $K_q=\rev{\beta}$ for the permutation knot $\beta$ obtained by reversing the orientation of $K_q$.  
\end{proof}

\begin{proof}[Proof of Proposition {\upshape\ref{prop:gammas-generate}}]
Denote by $\gen\subset\sk$ the sub-algebra $\sub{\gamma_m;m\in\Z}$.  So we want to argue that $\sk=\gen$.  

(1) It suffices to show that $\class{W}\in\gen$ for any web $W$.  Indeed, every element of $\sk$ can be written as a linear combination of classes of webs.  

(2)  It suffices to show that $\class{L}\in\gen$ for any link $L$.  Indeed, by iteratively applying the vertex removing relation, the class $\class{W}$ of a web can be written as a linear combination of classes of links.  

(3)  Assume by induction that $\class{L}\in\gen$ if $L$ has at most $N-1$ crossings.  The base case of a trivial link is immediate, using the unknot removing relation if need be.  Note that if $\class{L}\indeq\class{L^\prime}$, then  $\class{L}\in\gen$ if and only if $\class{L^\prime}\in\gen$.

It suffices to show that $\class{K}\in\gen$ for any knot $K$.  Indeed, by iteratively applying the crossing change relation, we have $\class{L}\indeq\class{K_1\cup K_2\cup\cdots\cup K_p}=\class{K_1}\class{K_2}\cdots\class{K_p}$ for knots $K_i$ lying at distinct heights.  

(4)  It suffices to show that $\class{\beta}\in\gen$ for any permutation knot $\beta=\lift{\sigma}\in\braid_m$ and its reverse $\rev{\beta}$.  Indeed, if the knot $K$ is contractible, then the relations imply $\class{K}\propto\class{\emptyset}\in\gen$.  Else, by Lemma \ref{lem:annulus-lemma}, there is a permutation knot $\beta=\lift{\sigma}\in\braid_m$ such that $\class{K}\indeq\class{\beta}$ or $\class{K}\indeq\class{\lp\rev{\beta}\rp}$.

(5)  We show that such a class $\class{\beta}$ or $\class{\lp\rev{\beta}\rp}$ as in Part (4) is contained in $\gen$, completing the proof.  By Lemma \ref{lem:isotopic-in-the-annulus}, the generator $\gamma_m$ is isotopic in the annulus to a mutant $\beta^{\prime}$ of $\beta$.  So $\class{\gamma_m}=\class{\beta^{\prime}}\in\sk$.  By applying the crossing change relation for each differing crossing, we have $\class{\beta^{\prime}}\indeq\class{\beta}$, in other words $\class{\gamma_m}\indeq\class{\beta}$.  Since $\class{\gamma_m}\in\gen$ by definition, we have $\class{\beta}\in\gen$ as desired.  By the second statement of Lemma \ref{lem:isotopic-in-the-annulus}, the same argument gives $\class{\lp\rev{\beta}\rp}\in\gen$.
\end{proof}

\begin{figure}[htb!]
\centering
\includegraphics[scale=0.8]{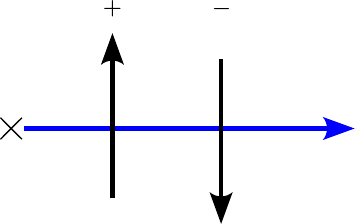}
\caption{Algebraic intersection number.}
\label{fig:fig7}
\end{figure}

Later we will use the following refinements.  

\begin{lemma}\label{lem:refinement}
Let $L$ be a link in $\ann\times\intt$ and let $\gamma$ be the projection of $L$ on $\ann$.  If $\gamma$ has at most $i$ positive intersections and $j$ negative intersections with an arc $\alpha$ oriented out of the puncture and cutting the annulus, colored blue 
in Figure {\upshape\ref{fig:fig7}}, then $\class{L}\in\sub{\gamma_m;-j\leq m\leq i}$.  
\end{lemma}

\begin{proof}
This essentially follows by the proof of Lemma \ref{lem:annulus-lemma}.  We argue by induction on the number of crossings of $L$.  The base case of a trivial link is immediate, using the unknot removing relation if need be.  Note that if $L$ satisfies the hypothesis of the lemma, namely has at most $i$ positive and $j$ negative intersections with $\alpha$, then both resolutions $L^\prime$ and $L^{\prime\prime}$ of $L$ obtained by performing a crossing relation also satisfy the hypothesis.  By iteratively applying the crossing change relation, write $\class{L}=\alpha\class{K_1}\class{K_2}\cdots \class{K_p}+\sum_i\alpha_i\class{L_i}$ for knots $K_i$, where the links $L_i$ have at most $N-1$ crossings and satisfy the hypothesis, so are covered by induction.  Note that the knots $K_i$ have $\leq i$ positive and $\leq j$ negative intersections with $\alpha$, so also satisfy the hypothesis.  We are reduced to the case of a non-contractible knot $K$ with $N$ crossings and with $i^\prime\leq i$ positive and $j^\prime\leq j$ negative intersections with $\alpha$.  We now run the explicit procedure in the proof of Lemma \ref{lem:annulus-lemma} to obtain a permutation  knot, where we may assume that the cut $\alpha$ is outside the box.  Note that throughout the procedure the number of positive or negative intersections with $\alpha$ never increases, in particular, any needed crossing change relations are covered by induction.   Repeat Part (5) of the proof of Proposition \ref{prop:gammas-generate}, where $-j^\prime\leq m\leq i^\prime$.
\end{proof}

\begin{lemma}\label{lem:leading-skein-term}
Let $\beta\in\braid_m$ be a braid link in $\ann\times\intt$ on $m$ strands.  There are Laurent polynomials $P_\beta(\qroot)$ and $\rev{P}_\beta(\qroot)$ in $\qroot$ such that $\beta-P_\beta(\qroot)\gamma_m\in\sub{\gamma_1,\gamma_2,\dots,\gamma_{m-1}}$ and $\rev{\beta}-\rev{P}_\beta(\qroot)\gamma_{-m}\in\sub{\gamma_{-1},\gamma_{-2},\dots,\gamma_{-(m-1)}}$ for all $\qroot\in\sC$.  Moreover, we can arrange that:  
\begin{enumerate}
\item  the Laurent polynomials $P_\beta(\qroot)=\rev{P}_\beta(\qroot)$ are the same;  
\item  if $\beta$ is connected, then the evaluation $P_\beta(\pm1)=1$ for $\qroot=\pm1$; 
\item  if $\beta$ is disconnected, then the evaluation $P_\beta(\pm1)=0$ for $\qroot=\pm1$.  
\end{enumerate}
\end{lemma}

\begin{proof}
The proof is by induction on the number $N$ of crossings of $\beta$.  The base case of a trivial braid link is immediate, where $P_\beta(\qroot)$ is either the constant polynomial $1$ or $0$ depending on whether $m=1$ or $m>1$, respectively.  Note that if $\beta$, $\beta_0$, $\beta_1$ are related by a crossing change relation such that $\beta=a_0\beta_0+a_1\beta_1$, where $\beta_0$ has $N$ crossings and $\beta_1$ has $N-1$ crossings, then both $\beta_0$ and $\beta_1$ are braid links on $m$ strands.  Note also that $a_0=1$ and $a_1=0$ when $\qroot=\pm1$. Thus, by the induction hypothesis it suffices to establish the result for $\beta_0$.  Iteratively using the crossing change relation to separate components of $\beta_0$ if there are multiple, it suffices to show it in the case that $\beta_0$ is connected.  As in the proof of Lemma \ref{lem:refinement}, we finish by running the explicit proof of Lemma \ref{lem:annulus-lemma}, followed by repeating Part (5) of the proof of Proposition \ref{prop:gammas-generate}.  

The existence of $\rev{P}_\beta(\qroot)$ is by reversing orientations throughout the argument, and the fact that we can arrange for $P_\beta(\qroot)=\rev{P}_\beta(\qroot)$ is because reversing all orientations preserves the coefficients $a_0$ and $a_1$ in the crossing change relation.  
\end{proof}

\begin{remark}\label{rem:remarkwithaquestion}
In Lemma \ref{lem:leading-skein-term}  we have established the existence of such polynomials.  In Section \ref{ssec:main-result-of-the-section} below we will prove a partial uniqueness statement.
\end{remark}

\subsection{Main result of the section for the bad set:  statement}\label{ssec:main-statement-of-the-section}

For each $i=1,2,\dots,n-1$ and $\sigma\in\sym_i$, choose Laurent polynomials $P_{\beta_{i,\sigma}}(\qroot)=\rev{P}_{\beta_{i,\sigma}}(\qroot)$ in $\qroot$ as in Lemma \ref{lem:leading-skein-term}, where the braid links $\beta_{i,\sigma}\in\braid_i$ are defined as in Section \ref{proof-of-the-last-statement-of-maintheorem}.  Define the Laurent polynomials
\begin{equation*}
P_i(\qroot)=\sum_{\sigma\in\sym_i}\lp-q^{(1-n)/n}\rp^{\ell(\sigma)}P_{\beta_{i,\sigma}}(\qroot).
\end{equation*}
Note we can take $P_1(\qroot)=1$.  

By Section \ref{proof-of-the-last-statement-of-maintheorem} and Lemma \ref{lem:leading-skein-term}, for the irreducible basis webs $B_i$, $1\leq i\leq n-1$,
\begin{equation*}\label{eq:rootsinduction1}\tag{$\ast$}
B_i-q^{n(n-1)}\lp\prod_{k=i}^{n-1}q^{-k}[n-k]\rp P_i(\qroot)\gamma_i\in\sub{\gamma_1,\gamma_2,\dots,\gamma_{i-1}},
\end{equation*}
\begin{equation*}\label{eq:rootsinduction2}\tag{$\ast\ast$}
B_{n-i}-q^{n(n-1)}\lp\prod_{k=i}^{n-1}q^{-k}[n-k]\rp P_i(\qroot)\gamma_{-i}\in\sub{\gamma_{-1},\gamma_{-2},\dots,\gamma_{-(i-1)}}.
\end{equation*}

\begin{lemma}\label{prop:nopm1roots}
For $\qroot=\pm1$, we have the evaluations 
\begin{equation*}
P_i(\pm1)=(-1)^{i-1}(\pm1)^{(1-n)(i-1)}(i-1)!
\end{equation*}
so that, in particular, $\qroot=\pm1$ are not roots of the Laurent polynomials $P_i(\qroot)$.  
\end{lemma}

\begin{proof}
By Lemma \ref{lem:leading-skein-term} the lower order terms in the sum vanish, hence
\begin{equation*}
P_i(\pm1)=\sum_{\{\sigma\in\sym_i;\sigma\text{ is an }i\text{-cycle}\}}\lp-(\pm1)^{(1-n)}\rp^{\ell(\sigma)}.
\end{equation*}
The result follows since there are $(i-1)!$ many $i$-cycles, each with signature $(-1)^{i-1}$, and $(-1)^{\ell(\sigma)}$ equals the signature of $\sigma$.
\end{proof}

Define the \define{bad set} $\altbadset\subset\sC$ as follows.  Let $\altbadset$ be the set $\badset$ as in Theorem \ref{thm:maintheorem} union the set of $\qroot\in\sC$ that are roots of at least one of the Laurent polynomials $P_{i}(\qroot)$ for $i=1,2,\dots,n-1$, rather, $2\leq i\leq n-1$ since $P_1(\qroot)=1$.  Note that, in particular, $\pm1\notin\altbadset$ by Lemma \ref{prop:nopm1roots}. The result we want to prove is:

\begin{proposition}\label{prop:gammastobasis}
For all $\qroot\in\sC\setminus\altbadset$, we have $\gamma_m\in\sub{B_1,B_2,\dots,B_{n-1}}$ for all $m\in\Z$.
\end{proposition}

We work toward the proof of Proposition \ref{prop:gammastobasis}, which appears in Section \ref{ssec:small-skeins}.  We will use the notation $\basgen=\sub{B_1,B_2,\dots,B_{n-1}}$.  

\subsection{Main result of the section for the bad set: induction step}\label{ssec:induction-step}

\begin{lemma}\label{lem:square-lemma}
Let $a=q^{n(n-1)}$, $b=-q^{(1-n)/n}$, and $a_k=q^{-k}[n-k]$.  For all $\qroot\in\sC$, the relation appearing in Figure {\upshape\ref{fig:square-lemma}} holds.  The same relation holds with all orientations reversed.\end{lemma}

\begin{proof}
This is a straightforward skein calculation, using the cilia change, vertex removing, and peeling relations.  The cilia change relation is optional.  
\end{proof}
\begin{figure}[htb!]
\centering
\includegraphics[width=\textwidth]{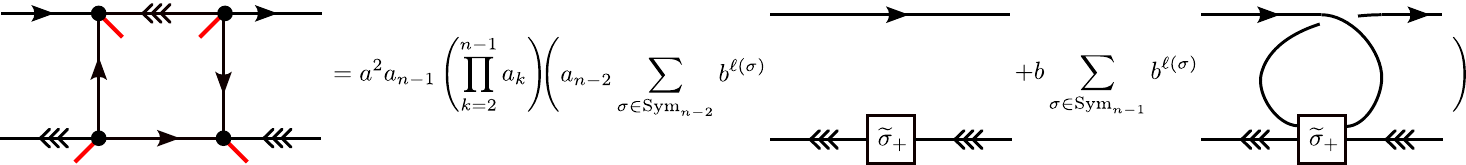}
\caption{Local square relation. A strand with a triple arrow indicates $n-2$ parallel strands oriented according to the direction of the arrow.}
\label{fig:square-lemma}
\end{figure}

\begin{lemma}\label{lem:inductionlargei}
Let $\qroot\in\sC$ such that $q$ is not a $(2k)$-root of unity for any $k=2,3,\dots,n-2$ but allowing $q=\pm1$.  Assuming $\gamma_i\in\basgen$ for all $|i|\leq n-1$, then $\gamma_i\in\basgen$ for all $|i|\geq n$.  
\end{lemma}

\begin{proof}
As a preliminary calculation, let $i\geq n$.  Using the planarizing relation, which is where the restriction on $q$ comes in, write $\gamma_i=cx+dy\in\sk$, where $c=q^{(n-1)/n}$ and $d=-q^{-n(n-1)/2-1/n}/[n-2]!$, and where $x$ and $y$ are the webs appearing in Figure \ref{fig:induction-1of3}. 
\begin{figure}[htb!]
\centering
\includegraphics[width=0.8\textwidth]{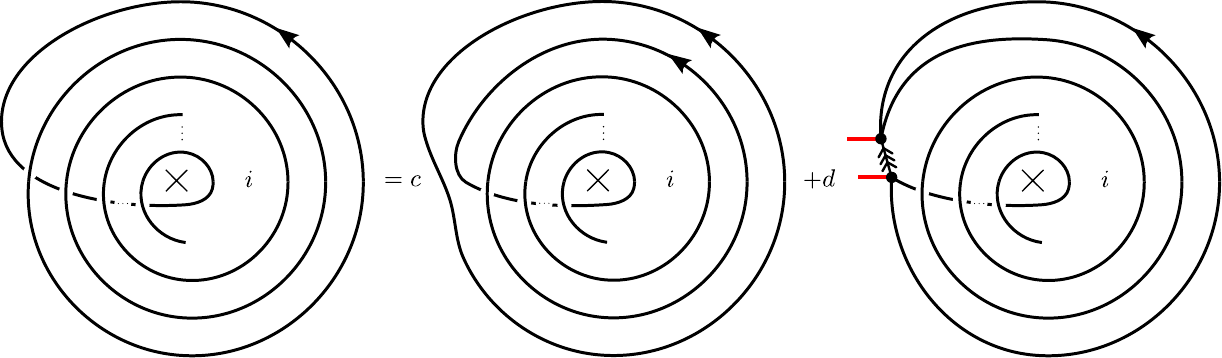}
\caption{The relation $\gamma_i=cx+dy$. A triple arrow indicates $n-2$ parallel strands.}
\label{fig:induction-1of3}
\end{figure}

We see that $x=\gamma_1\gamma_{i-1}\in\sk$.  By another use of the planarizing relation, write $y=cx^\prime+dy^\prime\in\sk$, where $x^\prime$ and $y^\prime$ are the webs appearing in Figure \ref{fig:induction-2of3}.  
 
\begin{figure}[htb!]
\centering
\includegraphics[width=0.6\textwidth]{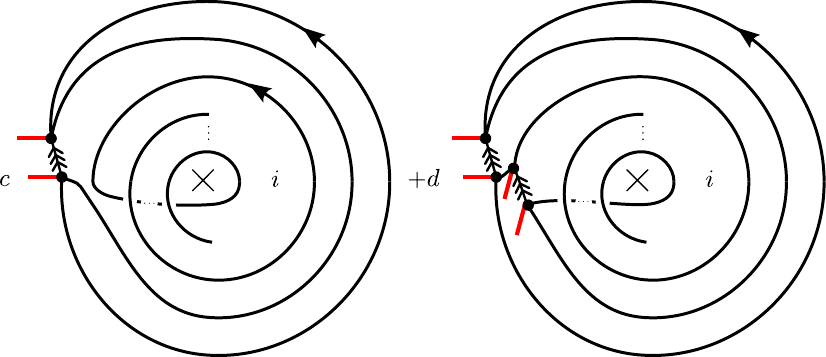}
\caption{The relation $y=cx^\prime+dy^\prime$.}
\label{fig:induction-2of3}
\end{figure}

We see that $x^\prime=B_2\gamma_{i-2}\in\sk$ or, in the particular case $n=2$, we have $x^\prime\propto\gamma_{i-2}\in\sk$.  Note that $y^\prime$ is isotopic to the web appearing on the left hand side of Figure \ref{fig:induction-3of3}.  
\begin{figure}[htb!]
\centering
\includegraphics[width=\textwidth]{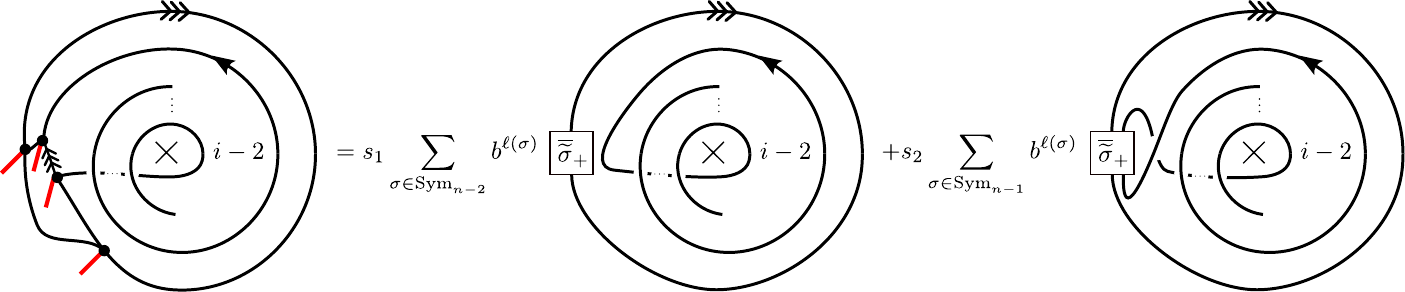}
\caption{The relation $y^\prime=s_1\sum_{\sigma\in\sym_{n-2}}b^{\ell(\sigma)}x^{\prime\prime}_\sigma+s_2\sum_{\sigma\in\sym_{n-1}}b^{\ell(\sigma)}y^{\prime\prime}_\sigma$.}
\label{fig:induction-3of3}
\end{figure}

Using Lemma \ref{lem:square-lemma}, write 
\begin{equation*}
y^\prime=s_1\sum_{\sigma\in\sym_{n-2}}b^{\ell(\sigma)}x^{\prime\prime}_\sigma+s_2\sum_{\sigma\in\sym_{n-1}}b^{\ell(\sigma)}y^{\prime\prime}_\sigma\in\sk
\end{equation*}
 where $s_1=a^2a_{n-1}a_{n-2}\prod_{k=2}^{n-1}a_k$ and $s_2=a^2a_{n-1}b\prod_{k=2}^{n-1}a_k$ with $a,b,a_k$ as in Lemma \ref{lem:square-lemma}, and where $x^{\prime\prime}_\sigma$ and $y^{\prime\prime}_\sigma$ are the webs appearing in Figure \ref{fig:induction-3of3}.  We see that $x^{\prime\prime}_\sigma=\rev{\lp\lift{\sigma}_+\rp}\gamma_{i-2}\in\sk$,
where recall the bar notation indicates the reverse orientation.  If $\alpha$ is a cutting arc as in Figure \ref{fig:fig7}, note by Lemma \ref{lem:refinement} that $\rev{\lp\lift{\sigma}_+\rp}\in\sub{\gamma_m;-(n-2)\leq m\leq0}\subset\basgen$, by hypothesis, and $y^{\prime\prime}_\sigma\in\sub{\gamma_m;-(n-2)\leq m\leq i-2}$. 

We now argue by induction on $|i|\geq n$.  For simplicity, assume $i>0$, the $i<0$ case following by a similar argument by reversing orientations and replacing $B_2$ above with $B_{n-2}$.  For the base case, $i=n$, by hypothesis and above we see that: $x=\gamma_1\gamma_{n-1}\in\basgen$; also $x^\prime=B_2\gamma_{n-2}\in\basgen$, or $x^\prime\propto\gamma_{n-2}=\gamma_0\in\basgen$ for $n=2$; also $x^{\prime\prime}_\sigma=\rev{\lp\lift{\sigma}_+\rp}\gamma_{n-2}\in\basgen$; and, $y^{\prime\prime}_\sigma\in\sub{\gamma_m;-(n-2)\leq m\leq n-2}\in\basgen$.  

For the induction step, assume in addition to the hypothesis that $\gamma_j\in\basgen$ for $j=n,n+1,\dots,i-1$. Similar to the base case, the preliminary calculation shows $\gamma_{i}\in\basgen$.  \end{proof}

\subsection{Main result of the section for the bad set:  finishing the proof}\label{ssec:small-skeins}

\begin{lemma}\label{lem:small-skeins}
Let $\altbadset\subset\sC$ be defined as in Section {\upshape\ref{ssec:main-statement-of-the-section}}, and let $\qroot\in\sC\setminus\altbadset$.  Then $\gamma_i\in\basgen$ for all $|i|\leq n-1$.  In fact, for $i=1,2,\dots,n-1$ we have $\gamma_i\in\sub{B_1,B_2,\dots,B_i}$ and $\gamma_{-i}\in\sub{B_{n-1},B_{n-2},\dots,B_{n-i}}$.    
\end{lemma}

\begin{proof}
To start, $\gamma_0=\class{\emptyset}\in\basgen$ by definition.  When in the positive position (Section \ref{proof-of-the-last-statement-of-maintheorem}), we have $B_1=(\prod_{k=1}^{n-1}a_k)\gamma_1\in\sk$ by the peeling relation, where $a_k$ is defined as in Lemma \ref{lem:square-lemma}.  Since $\qroot\notin\badset\subset\altbadset$ (Theorem \ref{thm:maintheorem}), we have $\prod_{k=1}^{n-1}a_k\neq0$, so, by inverting, $\gamma_1\in\sub{B_1}$.  Similarly, when in the negative position, the peeling relation gives $B_{n-1}=(\prod_{k=1}^{n-1}a_k)\gamma_{-1}\in\sk$, hence $\gamma_{-1}\in\sub{B_{n-1}}$.  Assume by induction that $\gamma_k\in\sub{B_1,B_2,\dots,B_k}$ for $k=1,2,\dots,i-1\leq n-2$.  Since $\qroot\notin\altbadset$, we can invert Equation \eqref{eq:rootsinduction1} to write $\gamma_i\in\sub{\gamma_1,\gamma_2,\dots,\gamma_{i-1},B_i}$.  By the induction hypothesis, we obtain $\gamma_i\in\sub{B_1, B_2, \dots, B_{i-1}, B_i}$ as desired.  Similarly, by induction, $\gamma_{-i}\in\sub{B_{n-1}, B_{n-2}, \dots, B_{n-(i-1)}, B_{n-i}}$ by inverting Equation \eqref{eq:rootsinduction2}.  
\end{proof}

\begin{proof}[Proof of Proposition {\upshape\ref{prop:gammastobasis}}]
Since $\qroot\notin\altbadset$, we have that $\gamma_i\in\basgen$ for all $|i|\leq n-1$.  Since $\qroot\notin\badset\subset\altbadset$, it follows by Lemma \ref{lem:inductionlargei} that $\gamma_i\in\basgen$ for all $|i|\geq n$.  So $\gamma_m\in\basgen$ for all $m\in\Z$ as desired.  
\end{proof}

\subsection{Main result of the section:  identifying the bad set}\label{ssec:main-result-of-the-section}

Let $\badset\subset\sC$ be as in Theorem \ref{thm:maintheorem}, and let $\altbadset\subset\sC$ be as in Section {\upshape\ref{ssec:main-statement-of-the-section}}.  Note that $\badset\subset\altbadset$.  In this section we prove:

\begin{lemma}\label{lem:twobadsetsarethesame}
We have $\altbadset=\badset$.  Equivalently, the roots $\qroot\in\sC$ of the Laurent polynomials $P_i(\qroot)$ are in $\badset$ for all $i=1,2,\dots,n-1$.
\end{lemma}

To prove the lemma, we will assume the linear independence property, Proposition \ref{prop:second-main-result-independence}, from Section \ref{sec:linear-independence} below.  Naturally then, the proof of this proposition will not depend on Lemma \ref{lem:twobadsetsarethesame} being true.  

\begin{proof}[Proof of Lemma {\upshape\ref{lem:twobadsetsarethesame}}]
By Equation \eqref{eq:rootsinduction1}, if $\qroot\in\sC$ is a root of $P_i(\qroot)$, then $B_i$ is in $\sub{\gamma_1,\gamma_2,\dots,\gamma_{i-1}}$.  Recall that $P_1(\qroot)=1$ and, by the proof of Lemma \ref{lem:small-skeins}, if $\qroot\notin\badset$ then $\gamma_1\in\sub{B_1}$.  Suppose $\qroot$ is a root not in $\badset$.  Let $2\leq i\leq n-1$ be minimal such that $\qroot$ is a root of $P_i(\qroot)$.  Again by the proof of Lemma \ref{lem:small-skeins}, for $1\leq j<i$ we then have $\gamma_j\in\sub{B_1,B_2,\dots,B_j}$.  It follows that $B_i\in\sub{\gamma_1,\gamma_2,\dots,\gamma_{i-1}}\subset\sub{B_1,B_2,\dots,B_{i-1}}$, which contradicts the linear independence property of Proposition \ref{prop:second-main-result-independence}.  
\end{proof}

Combining Proposition \ref{prop:gammas-generate}, Proposition \ref{prop:gammastobasis}, and Lemma \ref{lem:twobadsetsarethesame}, we immediately obtain the main result of the section:

\begin{proposition}\label{cor:basisgenerates}
For all $\qroot\in\sC\setminus\badset$, we have $\sk=\sub{B_1,B_2,\dots,B_{n-1}}$.  \qed
\end{proposition}

We end this sub-section with a brief aside, giving a partial answer to the question posed in Remark \ref{rem:remarkwithaquestion}.    

\begin{proposition}
Let $\beta\in\braid_i$ be a braid link in $\ann\times\intt$ on $i$ strands, $1\leq i\leq n-1$.  There exist unique Laurent polynomials $P_\beta(\qroot)$ and $\rev{P}_\beta(\qroot)$ in $\qroot$ such that $\beta-P_\beta(\qroot)\gamma_i\in\sub{\gamma_1,\gamma_2,\dots,\gamma_{i-1}}$ and $\rev{\beta}-\rev{P}_\beta(\qroot)\gamma_{-i}\in\sub{\gamma_{-1},\gamma_{-2},\dots,\gamma_{-(i-1)}}$ for all $\qroot\in\sC$.  In addition, these polynomials satisfy the properties (1), (2), (3) of Lemma {\upshape\ref{lem:leading-skein-term}}.  
\end{proposition}

\begin{proof}
Existence is by Lemma {\upshape\ref{lem:leading-skein-term}}.  We give the argument for $P_\beta(\qroot)$, the proof for $\rev{P}_\beta(\qroot)$ being essentially the same.  Let $P_\beta(\qroot)$ and $P^\prime_\beta(\qroot)$ be two such Laurent polynomials, and assume $\qroot\in\sC\setminus\badset$.  Then $\gamma_i(P_\beta(\qroot)-P^\prime_\beta(\qroot))\in\sub{\gamma_1,\gamma_2,\dots,\gamma_{i-1}}\subset\sub{B_1,B_2,\dots,B_{i-1}}$ by Lemma \ref{lem:small-skeins}.  Moreover, by the proof of Lemma \ref{lem:small-skeins}, $\gamma_i-a_i B_i\in\sub{B_1,B_2,\dots,B_{i-1}}$ for $a_i\neq0$.  If it were true that $P_\beta(\qroot)-P^\prime_\beta(\qroot)\neq0$, then it would follow that $B_i\in\sub{B_1,B_2,\dots,B_{i-1}}$ contradicting linear independence.  So $P_\beta(\qroot)=P^\prime_\beta(\qroot)$. As the functions are continuous and equal outside of the finite set $Q$, by continuity this equality holds for $\qroot\in\badset$ as well. 
\end{proof}

In light of this proposition, the Laurent polynomials $P_i(\qroot)$ for $i=1,2,\dots,n-1$ used to prove Proposition \ref{cor:basisgenerates} were indeed the natural candidates to consider.  Explicitly, for $i=1,2,3,4$ we computed that  
\begin{align*}
P_i(\qroot)&=(-1)^{i-1}q^{-(i-1)^2+(i-1)/n}\prod_{\zeta=(2j)\text{-root of unity}, \zeta\neq\pm1, 2\leq j\leq i-1}(q-\zeta)
\\&=(-1)^{i-1}q^{-i(i-1)/2+(i-1)/n}[i-1]!.
\end{align*}
In particular, for these computed examples, the roots of $P_i(\qroot)$ are in $\badset$ as required by Lemma \ref{lem:twobadsetsarethesame}.  We do not have a proof that the first equality above  continues to hold for  all~$i$.

As one last remark, if the above formula is indeed valid for all indices $i$, then it appears that it should not be too difficult, guided by the arguments of the present paper, to establish the main result in the more abstract setting where $q$ is treated as a formal variable and where some quantum integers are taken to be formally invertible.

\subsection{Digression:  countable generating set for general surfaces}\label{ssec:a-result-for-general-surfaces}

In this sub-section, we digress to prove a generalization of Proposition \ref{prop:gammas-generate}, which will not be used in the rest of the paper.  This is a very slight generalization of a well-known result for the HOMFLYPT skein algebra \cite{turaevMR0964255} of a surface.  

Let $\surf$ be an oriented surface, and let $\sk$ be the associated skein algebra (Remark \ref{rem:surfaces}).  For an element $\gamma\in\pi_1(\surf,x_0)$ of the fundamental group, the self intersection number $\iota(\gamma,\gamma)$ is the minimum number of self intersections of representatives, having only double points, of $\gamma$ in its homotopy class.  Recall that an element $\alpha\in\pi_1(\surf,x_0)$ of the fundamental group is primitive if $\alpha$ is not a power, namely, $\alpha=(\alpha^\prime)^k$ implies $k=1$.  Denote by $\Pi\subset\pi_1(\surf,x_0)$ the set of primitive elements.  For each primitive element $\alpha\in\Pi$, we choose a representative of $\alpha$ realizing its self intersection number $\iota(\alpha,\alpha)$, and by abuse of notation we also denote this representative by $\alpha$.  We can pick a solid torus $\tor_\alpha\subset\surf\times\intt$ whose projection on $\surf$ is a regular neighborhood of $\alpha$, that is, an immersed annulus.  In fact, there are $2^{\iota(\alpha,\alpha)}$ such choices, depending on crossing data.  For each such solid torus $\tor_\alpha\cong\ann\times\intt$ we have a copy $\gen_\alpha\cong\gen(\ann\times\intt)$ of the generating set for the skein algebra of the annulus.   

\begin{proposition}\label{prop:gensurf} 
For an oriented surface $\surf$, we have $\sk=\sub{\gen_\alpha;\alpha\in\Pi}$.  \qed
\end{proposition} 

We work toward the proof of this proposition.  Let $K$ be a knot in $\surf\times\intt$, and let $x\in\surf$ be a point  on the projection $\gamma$ of $K$ on the surface.  We say that $K$ is an ascending knot with respect to $x$ if, following orientation, the height of $K$ in $\surf\times\intt$ strictly increases, except over a neighborhood of $x$ where $K$ drops down from its maximum height to its lowest height at a point above $x$.  We may take this neighborhood in $\surf$ around $x$ to be arbitrarily small.  

\begin{lemma}\label{lem:ascending-knot}
Let $K$ be a knot in $\surf\times\intt$ with $N$ crossings, and let $x\in\surf$ be a point on the projection $\gamma$ of $K$ on the surface.  Then $\class{K}\indeq\class{K^\prime}$ where $K^\prime$ is an ascending knot with respect to $x$, and with the same projection $\gamma$ as $K$.  In particular, $K$ and $K^\prime$ have the same number of crossings.  
\end{lemma}

\begin{proof}
Starting at a point of $K$ above $x$, following orientation, lift the knot in the vertical direction, dropping at the end.  Apply crossing change relations if need be to pass the knot from below an impeding strand to above.  Note that the projection of the knot to the surface is unchanged throughout this process.
\end{proof}

\begin{lemma}\label{lem:can-take-minimal-crossings}
Let $K$ be a knot in $\surf\times\intt$ with $N$ crossings, with projection $\gamma$ of $K$ on the surface.  Then there is a knot $K^\prime$ with $\iota(\gamma,\gamma)$ crossings such that $\class{K}\indeq\class{K^\prime}$.  
\end{lemma}

\begin{proof}
If the number of crossings of $\gamma$ is strictly greater than $\iota(\gamma,\gamma)$, then by \cite{hassMR0804478} either $\gamma$ has an immersed monogon or an immersed bigon delineated by one or two arcs, respectively.  We will algorithmically remove these immersed bigons and monogons.  Choose such a bigon or monogon.  Let $x$ be a point in $\gamma$ not lying on the bigon or monogon.  Let $K^\prime$ be an ascending knot with respect to $x$ as in Lemma \ref{lem:ascending-knot}, which in particular has the same projection $\gamma$ as $K$.  Putting the drop of $K^\prime$ in a sufficiently small neighborhood over $x$, the bigon or monogon lies away from the drop.  In the case of a bigon, one arc lies entirely below the other, so by isotopy the bigon can be slid apart.  Similarly, in the case of a monogon, by isotopy the monogon can be shrunk into a neighborhood only containing the monogon, where it can be removed, at the cost of a non-zero scalar, by the kink removing relation.   Note, by the ascending condition, that there are no strands obstructing the isotopies in either case.  Repeat this process, possibly changing the point $x$ at each step, until there are no bigons or monogons.  
\end{proof}

\begin{lemma}\label{lem:isotopy-in-3-manifold}
Let $K$ be a knot in $\surf\times\intt$ with $N$ crossings.  The projection $\gamma$ of $K$ determines an element $\alpha^k$ in $\pi_1(\surf,x_0)$, up to conjugation, for some $\alpha\in\Pi$.  Then there is a knot $K^\prime$ lying in the solid torus $\tor_\alpha$ such that $\class{K}\indeq\class{K^\prime}$.  
\end{lemma}

\begin{proof}
We may assume that the basepoint $x_0$ lies on $\gamma$, and that $\alpha\in\Pi$ is chosen such that $\gamma=\alpha^k\in\pi_1(\surf,x_0)$.  Let $K^\prime$ be a knot as in Lemma \ref{lem:can-take-minimal-crossings}.  So $K^\prime$ has $N^\prime=\iota(\gamma,\gamma)\leq N$ crossings and $\class{K}\indeq\class{K^\prime}$.  We may assume that the projection $\gamma^\prime$ of $K^\prime$ contains the basepoint $x_0$.  By Lemma \ref{lem:ascending-knot}, we may also assume that $K^\prime$ is ascending with respect to $x_0$, with an arbitrarily small neighborhood over $x_0$ containing the drop.  By the construction of $K^\prime$, we have $\alpha^k=\gamma=\gamma^\prime\in\pi_1(\surf,x_0)$.  Since $\gamma^\prime$ achieves its self intersection number, by \cite{hassMR1216633} the homotopy $\gamma_t$ from $\gamma^\prime$ to $\alpha^k$ can be taken such that throughout the homotopy the crossing number of $\gamma_t$ is constant equal to $N^\prime$.  Assume for the moment that $\iota(\alpha,\alpha)=0$, so that the solid torus $\tor_\alpha$ is uniquely chosen.  Lift the homotopy $\gamma_t$ to a constant height homotopy $K_t$ of the ascending knot $K^\prime$, fixing the drop within a small neighborhood over $x_0$.  This lifted homotopy $K_t$ can be realized by a sequence of isotopies and crossing change relations, the latter corresponding to finitely many times $t_i$ where $K_t$ passes through the drop.  Since the crossing number $N^\prime$ is constant throughout the homotopy, we have $\class{K_{t_i}}\altindeq\class{K_{t_{i+1}}}$.  At the end of the homotopy, $K_1$ lies in the solid torus $\tor_\alpha$ and $\class{K^\prime}\altindeq\class{K_1}$, implying $\class{K^\prime}\indeq\class{K_1}$ as desired.  In the case that $\iota(\alpha,\alpha)\neq0$, if part of $K_1$ happens to lie in one of the other solid tori lifting $\alpha$, then $K_1$ can be corrected into the solid torus $\tor_\alpha$ by crossing change relations.  
\end{proof}

\begin{proof}[Proof of Proposition {\upshape\ref{prop:gensurf}}]
Put $\gen=\sub{\gen_\alpha;\alpha\in\Pi}$.  Identically as in the proof of Proposition \ref{prop:gammas-generate}, arguing by induction on the number of crossings, it suffices to show that $\class{K}\in\gen$ for any knot $K$ in $\surf\times\intt$ with $N$ crossings.  By Lemma \ref{lem:isotopy-in-3-manifold}, there is a knot $K^\prime$ in the solid torus $\tor_\alpha$ for some primitive $\alpha\in\Pi$ such that $\class{K}\indeq\class{K^\prime}$.  By Proposition \ref{prop:gammas-generate}, we have $\class{K^\prime}\in\gen_\alpha\subset\gen$.
\end{proof}

\begin{remark}
 Proposition \ref{prop:gensurf} provides an explicit countably infinite generating set of the skein algebra~$\sk$.  
\end{remark}

\section{Linear independence: the quantum trace map}\label{sec:linear-independence}

In this section, we prove the linear independence property in Theorem \ref{thm:maintheorem} as an application of the quantum trace map:

\begin{proposition}\label{prop:second-main-result-independence}
Let $\badset\subset\sC$ be defined as in Theorem {\upshape\ref{thm:maintheorem}}.  For all $\qroot\in\sC\setminus\badset$, the monomials of the form $B_{1}^{k_1}B_{2}^{k_2}\cdots B_{n-1}^{k_{n-1}}\in\sk$ for $k_i\in\nnegZ$ are linearly independent in $\sk$.
\end{proposition}

We work toward the proof of Proposition \ref{prop:second-main-result-independence}, which appears in Section \ref{sssec:annuluscomputationfinishing}. Let $\surf$ be an oriented finite-type surface, possibly with nonempty boundary $\partial\surf$.  Our application will be in the case when $\surf$ is the annulus.  

\begin{remark}
Much of the following theory should be adaptable to infinite-type surfaces  without too much effort.
\end{remark}

In Section \ref{ssec:stated-webs}, we define stated webs in any thickened surface.  In Section \ref{ssec:stated-skeins}, we define the stated skein algebra for any surface.  In Section \ref{ssec:punctured-surfaces}, we define punctured surfaces and explain how to assign a number to any skein in the thickened biangle, called the biangle quantum trace of the skein.  In Sections:  \ref{sssec:good-position-of-a-web-with-respect-to-a-split-ideal-triangulation}, we define ideally triangulated punctured surfaces; \ref{sssec:quantum-trace-map}, we preliminarily introduce the quantum trace map as the assignment of an element of a certain non-commutative algebra, coming from Fock--Goncharov theory and depending on the ideal triangulation, to each skein in a thickened ideally triangulated punctured surface; \ref{split-ideal-triangulations-good-position-and-definition-of-the-quantum-trace-map}, preparing to give a definition of the quantum trace, we describe the technical notion of a good position of a web with respect to the split ideal triangulation; \ref{sssec:definition-of-the-quantum-trace-map}, we define the quantum trace as a state sum with respect to the split ideal triangulation; \ref{ssec:fock-goncharov-quantum-torus}, we, after the fact, provide the necessary details about the Fock--Goncharov non-commutative algebra; \ref{sssec:quantum-traces-from-spectral-networks}, we digress to explain an open mathematical problem about quantum traces originating in the physical point of view.  In Sections:  \ref{sssec:annuluscomputationsetup}, preparing to prove Proposition \ref{prop:second-main-result-independence}, we provide the topological setup for the annulus calculation; \ref{sssec:annuluscomputationsomemoregeneralities}, we perform some preliminary calculations; \ref{sssec:annuluscomputationfinishing} we finish the proof of Proposition \ref{prop:second-main-result-independence}.  

\subsection{Stated webs}\label{ssec:stated-webs}

It will be useful to generalize the notion of web in $\surf\times\intt$ (Section \ref{ssec:webs} and Remark \ref{rem:websfor3manifolds}) as follows.  A \define{based boundary web}, also denoted $W$, is defined exactly as a web, except in addition we allow for oriented arcs ending on the boundary $(\partial\surf)\times\intt$.  In particular, $\partial W=W\cap((\partial\surf)\times\intt)$, and we allow $\partial W=\emptyset$.  The framing of such arcs at points in $W\cap((\partial\surf)\times\intt)$, which we think of as monovalent vertices of $W$, is required to be pointed in the positive $\intt$ direction.  Note that there is no cilia data attached to monovalent vertices.  An important other technical condition is that for individual components $\comp$ of $\partial\surf$, we require that the height values in $\intt$ of the points $W\cap(\comp\times\intt)$ in a boundary wall are distinct.  Boundary webs are considered up to isotopy through the family of boundary webs.  In particular, isotopies cannot remove points from a boundary, and cannot exchange the vertical heights of strands of $W$ ending on the same boundary wall $\comp\times\intt$.  Similarly, one defines the notion of an \define{unbased boundary web}, also denoted $\unbased{W}$.  

A \define{state} $s$ for a based boundary web $W$ is a function $s:W\cap((\partial\surf)\times\intt)\to\{1,2,\dots,n\}$.  A based boundary web equipped with a state is called a \define{stated based boundary web}.  We will often just write $W$ to indicate a stated based boundary web $(W,s)$, and will often confuse the word `state' with the value of a state at a given boundary point.  Stated based boundary webs are considered up to isotopy through the family of stated based boundary webs.  Similarly, one defines the notion of a \define{stated unbased boundary web}.  

Actually, in contrast to Sections \ref{sec:mainresult} and \ref{sec:spanning}, for technical reasons in this section we work primarily with stated unbased boundary webs $\unbased{W}$, which we will simply call \define{stated webs}.  

In order to represent a stated web, rather, its isotopy class, by a diagram, we need to indicate in the diagram the height data associated to points $\unbased{W}\cap(\comp\times\intt)$ of the web in each boundary wall.  This can be achieved by drawing a white arrow on the boundary component $\comp$, together with a starting location, such as in Figure \ref{fig:sikora-le-boundary-relations}, where the heights of the boundary points increase as one moves in the direction of the arrow.  As an example, in relation (B) of Figure \ref{fig:sikora-le-boundary-relations}, the point with state $j$ is higher than the point with state $i$, where in the figure we have assumed the starting location for the arrow on the boundary component $\comp$ is to the right.  Note that such diagrams always exist, by horizontally sliding via isotopy the boundary points of $\unbased{W}$ along a boundary wall.  Note also that such an arrow is only necessary when $\unbased{W}\cap(\comp\times\intt)$ consists of at least two points.    Soon, we will work only with punctured surfaces (Section \ref{ssec:punctured-surfaces}), where the starting location is determined by the arrow so is superfluous.  

\subsection{Stated skeins}\label{ssec:stated-skeins} 

Following \cite{LeArxiv21}, the \define{stated skein module} $\ssk(\surf)=\ssk(\surf\times\intt)$ is the vector space obtained by quotienting the vector space of formal finite linear combinations of isotopy classes of stated webs $\unbased{W}$ in $\surf\times\intt$ by the internal relations shown in Figure \ref{fig:sikora-le-relations}, together with the boundary relations shown in Figure \ref{fig:sikora-le-boundary-relations} as well as the boundary relations obtained from those in the figure by reversing all strand orientations. 

\begin{figure}[htb!]
\centering
\includegraphics[width=0.99\textwidth]{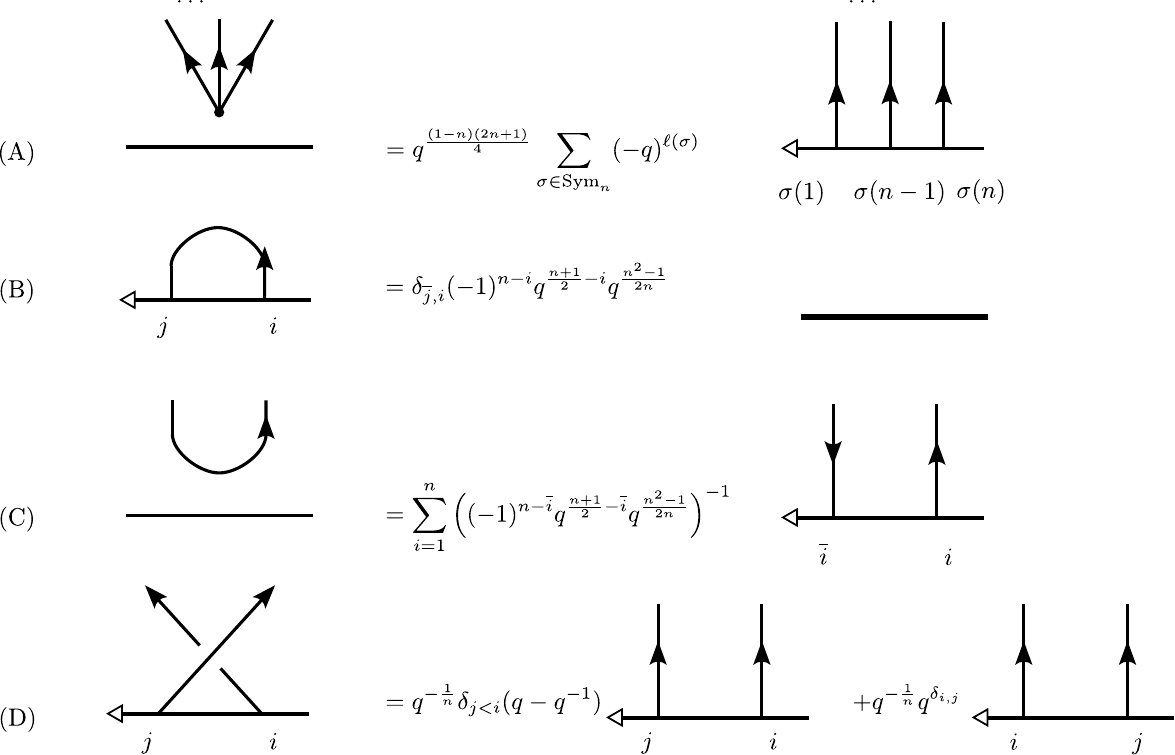}
\caption{Defining boundary relations for the stated skein algebra, including reversing all strand orientations.}
\label{fig:sikora-le-boundary-relations}
\end{figure}

In the figure: the states $i,j,\bar{i},\sigma(k)\in\{1,2,\dots,n\}$ are indicated on the boundary; we have $\bar{i}=n+1-i$; and, we have $\delta_{i,j}$ is $1$ if $i=j$ and $0$ else, and $\delta_{j<i}$ is $1$ if $j<i$ and $0$ else.  As before, note that the stated skein module $\ssk(\surf)$ actually depends on $\qroot$.  

Also as before, we sometimes use the notation $\class{\unbased{W}}$ to express the class in the skein module of a stated web $\unbased{W}$.  And, the skein module has an, in general non-commutative, algebra structure by stacking, using the same convention as earlier, so we refer to $\ssk(\surf)$ as the \define{stated skein algebra}.  

\begin{remark}
In \cite{LeArxiv21}, stated webs $\unbased{W}$ and the stated skein module $\ssk(\man)$ are defined for more general oriented three-dimensional manifolds with boundary $\man$ (Remark \ref{rem:surfaces}).
\end{remark}

\subsection{Punctured surfaces and the biangle quantum trace map}\label{ssec:punctured-surfaces}

The surface $\surf$ is \define{punctured} if it is obtained from a compact finite-type surface $\surf^\prime$ by removing a finite set $\punc$ from $\surf^\prime$.  That is, $\surf=\surf^\prime\setminus\punc$.  We require in addition that $\punc$ intersects every component of $\surf^\prime$ and every component of the boundary $\partial \surf^\prime$. Consequently, every boundary component $\comp$ of the punctured surface $\surf$ is an \define{ideal arc}, namely, homeomorphic to  an open interval.  In particular, in our diagrammatic notation for stated webs, we only need to provide just the arrow on the boundary component $\comp$, since the starting location is then determined in the obvious way, namely in the direction of the tail of the arrow.  

An \define{ideal monoangle} $\monoangle$, resp. \define{ideal biangle} $\biang$ or \define{ideal triangle} $\ttriang$ (or just \define{monoangle}, resp. \define{biangle} or \define{triangle}) is the punctured oriented surface obtained by removing one, resp. two or three, points from the boundary of a closed disk with the standard orientation.  

\begin{theorem}[\cite{costantinoMR4493620,LeArxiv21}]\label{thm:monoanglebiangletriangle}
The stated skein algebra $\ssk(\monoangle)$ of the monoangle is isomorphic as an algebra to  $\C$.  The stated skein algebra $\ssk(\biang)$ of the biangle is a Hopf algebra isomorphic to the quantized coordinate ring $\mcal{O}_q(\nsl)$.  The stated skein algebra $\ssk(\ttriang)$ of the triangle is isomorphic to the braided tensor product of $\mcal{O}_q(\nsl)$ with itself.  
\end{theorem}

For terminology and details, see \cite{costantinoMR4493620,LeArxiv21}.  Let us now record the precise details from this theorem that we will require.  Note that the isomorphism $\ssk(\monoangle)\cong\C$ is obtained simply by applying the skein relations to write any skein $\sum_i a_i \class{\unbased{W}_i} = a \class{\emptyset}$, the number $a\in\C$ being the value of the skein under the isomorphism.  What will be useful for us later on is that since the stated skein algebra $\ssk(\biang)$ of the biangle is a Hopf algebra, in particular it comes equipped with a \define{co-unit} algebra homomorphism 
\begin{equation*}
\counit:\ssk(\biang)\to\ssk(\monoangle)\cong\C.
\end{equation*}  
Another name for the co-unit $\counit$ is the \define{biangle quantum trace map}.  

In Figure \ref{fig:counit-def}, we provide the diagrammatic definition of the co-unit $\counit(\class{\unbased{W}})$ evaluated on the class of a stated web $\unbased{W}$.  In the figure, on the right hand side, the upper puncture has been filled in, yielding a stated web in a monogon, which is a number, the value of the co-unit, as just discussed.  Note also that the height order is reversed on the right edge.  The definition for a general skein $\sum_i a_i \class{\unbased{W}_i}$ is by linear extension.  

\begin{figure}[htb!]
\centering
\includegraphics[width=\textwidth]{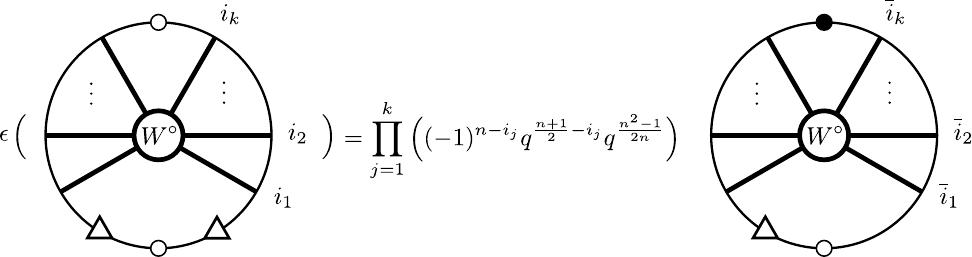}
\caption{Definition of the co-unit.}
\label{fig:counit-def}
\end{figure}

For $k\in\{0,1,\dots,n\}$ and $i_j\in\{1,2,\dots,n\}$ for $j=1,2,\dots,k$ and $l_{j^\prime}\in\{1,2,\dots,n\}$ for $j^\prime=1,2,\dots,n-k$, let $\source(k,(i_j)_j,(l_{j^\prime})_{j^\prime})$ and $\sink(k,(i_j)_j,(l_{j^\prime})_{j^\prime})$ be the \define{source} and \define{sink} stated webs in the biangle $\biang$ appearing in Figure \ref{fig:counit-example}.  The co-unit $\counit$ has the same values on these two stated webs, which we provide in the figure.  Here, if $\{i_1,i_2,\dots,i_k,\bar{l}_{n-k},\dots,\bar{l}_2,\bar{l}_1\}$ has less than $n$ elements then the co-unit is zero, else $\bmat i_1&i_2&\hdots&i_k&\bar{l}_{n-k}&\hdots&\bar{l}_{2}&\bar{l}_1\emat\in\sym_n$ is the permutation determined by the boundary states in this order.  

Later, we will need the following formula in Lemma \ref{lem:summingcounitoverthesymmetricgroup}.  Figure \ref{fig:triangulated-annulus} below might be helpful in keeping track of the indices.

\begin{lemma}\label{lem:summingcounitoverthesymmetricgroup}
For $k$, $(i_j)_j$, $(l_{j^\prime})_{j^\prime}$ as above, assume $i_1<i_2<\dots<i_k$ and $l_{n-k}>\dots>l_2>l_1$ and $i_j\neq \bar{l}_{j^\prime}$ for all $j,j^\prime$.  Then
\begin{gather*}
\sum_{\sigma^\prime\in\sym_{n-k},\sigma\in\sym_{k}}
\counit(\sink(n-k,(l_{\sigma^\prime(j^\prime)})_{j^\prime},(i_{\sigma(j)})_j))\counit(\source(k,(i_{\sigma(j)})_j,(l_{\sigma^\prime(j^\prime)})_{j^\prime}))
\\=q^{(n-k)(n-k-1)/2}[n-k]!q^{k(k-1)/2}[k]!
\counit(\sink(n-k,(l_{j^\prime})_{j^\prime},(i_{j})_j))\counit(\source(k,(i_{j})_j,(l_{j^\prime})_{j^\prime})).
\end{gather*}
\end{lemma}

\begin{proof}
Note that $\bar{l}_{n-k}<\dots<\bar{l}_2<\bar{l}_1$.  In particular, $(i_1,i_2,\dots,i_k,\bar{l}_{n-k},\dots,\bar{l}_2,\bar{l}_1)\in\sym_n$ has minimal length among the $(i_{\sigma(1)},i_{\sigma(2)},\dots,i_{\sigma(k)},\bar{l}_{\sigma^\prime(n-k)},\dots,\bar{l}_{\sigma^\prime(2)},\bar{l}_{\sigma^\prime(1)})$.  Similarly for $(l_1,l_2,\dots,l_{n-k},\bar{i}_k,\dots,\bar{i}_2,\bar{i}_1)$.  

\begin{figure}[htb!]
\centering
\includegraphics[width=.833\textwidth]{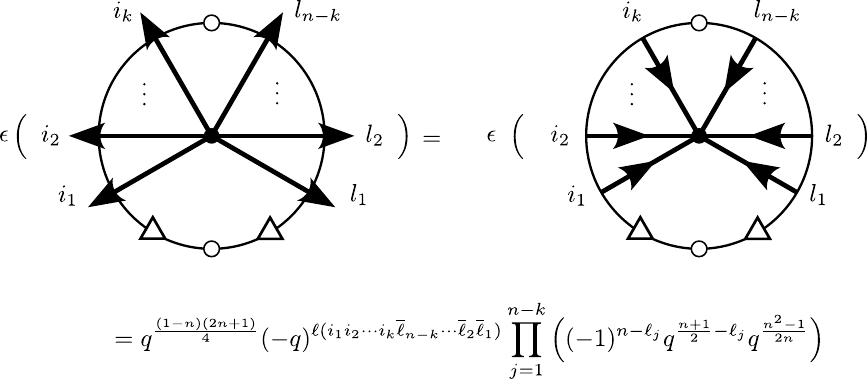}
\caption{Co-unit for source and sink webs in the biangle.}
\label{fig:counit-example}
\end{figure}

By the formula for the co-unit of the sink and source webs in Figure \ref{fig:counit-example}, the left hand side of the equation in the statement is equal to
\begin{equation*}
\sum_{\sigma^\prime\in\sym_{n-k},\sigma\in\sym_{k}}
(-q)^{2\ell(\sigma^\prime)}(-q)^{2\ell(\sigma)}
\counit(\sink(n-k,(l_{j^\prime})_{j^\prime},(i_{j})_j))\counit(\source(k,(i_{j})_j,(l_{j^\prime})_{j^\prime}))
\end{equation*}
which is equal to the right hand side by the general formula 
$\sum_{\sigma\in\sym_k}(-q)^{2\ell(\sigma)}=q^{k(k-1)/2}[k]!$; see, for example, \cite[Equation (21)]{SikoraAlgGeomTop05}.
\end{proof}

\subsection{Ideally triangulated punctured surfaces and the quantum trace map}\label{ssec:ideally-triangulable-punctured-surfaces-and-the-quantum-trace-map}

From now on, $\surf=\surf^\prime-\punc$ will be a punctured surface.  

\subsubsection{Ideal triangulations}\label{sssec:good-position-of-a-web-with-respect-to-a-split-ideal-triangulation}

An \define{ideal triangulation} $\idealtriang$ of the punctured surface $\surf$ is a cell decomposition of $\surf^\prime$ such that the $0$-cells consist of the puncture set $\punc$, the $1$-cells are ideal arcs, and the $2$-cells are ideal triangles. It is well-known \cite{bonahonMR2851072} that if $\surf$ has $d\geq0$ boundary components, then an ideal triangulation $\idealtriang$ exists if and only if the Euler characteristic of $\surf$ is strictly less than $d/2$. Note that the only punctured surfaces excluded in this way are the monoangle, the biangle, and the sphere with one or two punctures.  For simplicity, assume ideal triangulations do not contain any self-folded triangles, in which case the once punctured monoangle should also be excluded.  

The $1$-cells, namely the ideal arcs, often called ideal edges, of an ideal triangulation $\idealtriang$ are partitioned into the internal edges and the boundary edges, the latter which are exactly the boundary components $\comp$ of $\surf$.  The $2$-cells are called either triangles or faces.  

\subsubsection{Quantum trace map:  introduction}\label{sssec:quantum-trace-map}

Let $\surf$ be an ideally triangulated surface with ideal triangulation $\idealtriang$.  Fix once and for all an arbitrary $(2n)$-root $\wroot^{1/2}$ of $\qroot$.  So $(\wroot^{1/2})^{2n}=\qroot$. Arbitrarily identify each ideal triangle of $\idealtriang$ with the ideal triangle $\ttriang$, then choose an arbitrary order of the triangles $\ttriang$ of $\idealtriang$.  

In Section \ref{ssec:fock-goncharov-quantum-torus} below, we will associate a non-commutative algebra $\qtorus{\ttriang}$, called the \define{$n$-root Fock--Goncharov quantum torus} (or just \define{quantum torus}), to the ideal triangle $\ttriang$, rather, to $\ttriang$ plus the choice of a preferred ideal vertex of this triangle, whose choice is only for bookkeeping purposes and which we will likely not mention again.  

Recall that if $\mcal{A}$ and $\mcal{B}$ are algebras, then their tensor product $\mcal{A}\otimes\mcal{B}$ over $\C$ is also an algebra, where the product of $a\otimes b$ with $a^\prime\otimes b^\prime$ is $(aa^\prime)\otimes(bb^\prime)$.  Let $\bigotimes_{\ttriang\in\idealtriang}\qtorus{\ttriang}$ be the tensor product algebra of the quantum tori $\qtorus{\ttriang}$ associated to the triangles $\ttriang$, in the given order.  Note that, since the algebras $\mcal{A}\otimes\mcal{B}$ and $\mcal{B}\otimes\mcal{A}$ are naturally isomorphic in the obvious way, this order is essentially immaterial.   

\begin{theorem}[\cite{bonahonMR2851072} for $n=2$; \cite{douglasMR4717274, kim2022rm} for $n=3$; \cite{le2023quantum} for $n=n$]
For an ideally triangulated surface $\surf$ with ideal triangulation $\idealtriang$, and a $(2n)$-root $\wroot^{1/2}$ of $\qroot$, there exists a unique algebra homomorphism
\begin{equation*}
\qtr:\ssk(\surf)\to\bigotimes_{\ttriang\in\idealtriang}\qtorus{\ttriang}
\end{equation*}
from the stated skein algebra of the surface to the tensor product, over the triangles of $\idealtriang$, of the $n$-root Fock--Goncharov quantum tori satisfying certain enjoyable properties.  
\end{theorem}

This is the \define{quantum trace map}.  See Section \ref{sssec:definition-of-the-quantum-trace-map} below for the definition.  In fact, the image of $\qtr$ lies in a nice sub-algebra of the tensor product, but we will not need this.

One such enjoyable property is that when the $(2n^2)$-root $\wroot^{1/2}$ of $q$ equals $1$, then the quantum trace map expresses the trace of the monodromy of a framed $\mrm{PSL}_n(\C)$ local system on $\surf$ as a Laurent polynomial in $n$-roots of the Thurston--Fock--Goncharov $\mcal{X}$-coordinates; see \cite{bonahonMR2851072, douglasMR4717274}.  Another property is that the quantum trace map is natural, in a suitable sense, with respect to the choice of ideal triangulation $\idealtriang$.  

When the surface $\surf$ has empty boundary, $\partial\surf=\emptyset$, it has been proved that the quantum trace map is injective for $n=2$ \cite{bonahonMR2851072} and $n=3$ \cite{kim2022rm}.  The proof uses the fact that explicit linear bases of $\sk$ are known in these cases.  It is expected to be true for $n=n$ as well.  As an application of our main theorem, which provides such a basis when $\surf=\ann$ is the annulus, we establish this injectivity property of the $\nsl$ quantum trace map for the annulus, in Section \ref{ssec:injectivity} below.  

\subsubsection{Good position of a web with respect to a split ideal triangulation}\label{split-ideal-triangulations-good-position-and-definition-of-the-quantum-trace-map}

The \define{split ideal triangulation} $\splitidealtriang$ associated to an ideal triangulation $\idealtriang$ is the cell complex obtained by blowing up each edge of $\idealtriang$ to a biangle $\biang$, including the boundary edges.  In total, the number of $1$-cells doubles, and there is a new $2$-cell introduced for each prior edge.  Each triangle $\ttriang$ then faces three biangles, each internal biangle faces two triangles, and each external biangle, namely those coming from boundary edges, faces one triangle.  Note that there is a natural bijection between the triangles of $\idealtriang$ and those of $\splitidealtriang$.  

We now describe what we mean by the \define{good position} of a stated web $\unbased{W}$ in $\surf\times\intt$ with respect to a split ideal triangulation $\splitidealtriang$ of $\surf$.  Note this is a purely topological notion, and the states themselves will not play any role.  The web $\unbased{W}$ is in good position if it satisfies the following properties:  (1)  the restriction of $\unbased{W}$ to $\biang\times\intt$ or $\ttriang\times\intt$ for each thickened face, either a biangle $\biang$ or a triangle $\ttriang$, of $\splitidealtriang$ is a boundary web in the sense of Section \ref{ssec:stated-webs};  (2)  the restriction of $\unbased{W}$ to $\ttriang\times\intt$ for each triangle face is a disjoint union of oriented arcs $\unbased{A}_1, \unbased{A}_2, \dots, \unbased{A}_k$, each at a constant height with a constant vertical framing and connecting two different edges of the triangle.  Note by (1) that, over each triangle, the arcs are at different heights, and their framings point in the positive vertical direction.  We may assume the indexing is such that the arc $\unbased{A}_i$ lies above $\unbased{A}_{i+1}$.  

\begin{remark}
Note that, in the case $n=3$, this notion of good position for a web with respect to a split ideal triangulation $\splitidealtriang$ is different than that used in \cite{douglasMR4685684, douglas2022tropical}.  
\end{remark}

Mimicking the proof in \cite{bonahonMR2851072}, good positions always exist, and two good positions can be related by a sequence of Reidemeister-like moves; see also \cite{douglasMR4717274, kim2022rm}. 

Let $\unbased{W}$ be a stated web in $\surf\times\intt$, and $\alpha$ a nonboundary ideal arc in $\surf$, namely $\alpha$ connects two punctures.  We allow $\alpha$ to be isotopic to a boundary edge, but not intersecting.  Assume the points of $W\cap(\alpha\times\intt)$, if any, are at distinct heights and with blackboard framing.  Then we can cut along the arc $\alpha$ to obtain a new punctured surface $\surf_\alpha$ with two new boundary edges $\alpha_1$ and $\alpha_2$, as well as a new boundary web $\unbased{W}_\alpha$ in $\surf_\alpha\times\intt$, which is not naturally stated.  A state $s_\alpha$ for $\unbased{W}_\alpha$ is called \define{compatible} if its values $s_{\alpha}(p_1)=s_{\alpha}(p_2)$ agree on pairs of endpoints $p_1\in\alpha_1\times\intt$ and $p_2\in\alpha_2\times\intt$ lying over the cut arc that were originally attached before the cutting, and otherwise $s_\alpha$ agrees with the state $s$ of $\unbased{W}$.  

\subsubsection{Quantum trace map:  definition, part 1 of 2}\label{sssec:definition-of-the-quantum-trace-map}

As a first case, suppose $\surf=\ttriang$ is the ideal triangle, and $\unbased{W}=\unbased{A}$ is a flat oriented stated arc between two edges of the triangle, with the blackboard framing.  Let $s=(i,j)$ be the state of $\unbased{A}$, where $i$ (resp. $j$) in $\{1,2,\dots,n\}$ is associated to the tail (resp. tip) of the arc.  In Section \ref{ssec:fock-goncharov-quantum-torus} below, we will assign an element $M(\unbased{A},s)=M(\unbased{A})_{ij}\in\qtorus{\ttriang}$ for each state $s$.  This is the $ij$-entry of the \define{Fock--Goncharov quantum local monodromy matrix} $M(\unbased{A})$ associated to the arc $\unbased{A}$, namely the entry in the $i$-th row and $j$-th column.  

For a general surface $\surf$, let $\unbased{W}$ be a stated web in good position with respect to the split ideal triangulation $\splitidealtriang$.  For each triangle $\ttriang$ of $\splitidealtriang$, let $\unbased{A}_1, \unbased{A}_2, \dots, \unbased{A}_k$ be the corresponding arcs, as in Section \ref{split-ideal-triangulations-good-position-and-definition-of-the-quantum-trace-map}, ordered from higher to lower.  For each biangle $\biang$ of $\splitidealtriang$, let $\unbased{E}$ denote the restriction of $\unbased{W}$ to the thickened biangle.  The $\unbased{A}_i$ and $\unbased{E}$ are boundary webs, but are not necessarily stated.  As in the end of Section \ref{split-ideal-triangulations-good-position-and-definition-of-the-quantum-trace-map}, for a given choice of states, we can talk about whether the stated webs $(\unbased{A}_i,s_i)$ and $(\unbased{E},s)$ are compatible, varying over all triangles and biangles.  The quantum trace map is then defined by the \define{state sum formula}
\begin{equation*}
\qtr(\unbased{W})=\sum_{\text{compatible states}}\lp\prod_{\biang\in\splitidealtriang}\counit(\unbased{E},s)\rp
\lp\bigotimes_{\ttriang\in\splitidealtriang}M(\unbased{A}_1,s_1)\cdots M(\unbased{A}_k,s_k)\rp
\in\bigotimes_{\ttriang\in\idealtriang}\qtorus{\ttriang}.
\end{equation*}
Here, $\counit$ is the biangle quantum trace map defined in Section \ref{ssec:punctured-surfaces}.  Note that since the quantum tori $\qtorus{\ttriang}$ are non-commutative algebras, the order of multiplication of the $M(\unbased{A}_i,s_i)$ of the arc terms in the above formula is important.  This order is controlled by the heights of the arcs lying over the triangles.  

\subsubsection{Quantum trace map:  definition, part 2 of 2}\label{ssec:fock-goncharov-quantum-torus}

We now, after the fact, define the $n$-root Fock--Goncharov quantum torus $\qtorus{\ttriang}$ associated to the ideal triangle $\ttriang$, as well as the Fock--Goncharov quantum local monodromy matrices $M(\unbased{A})=(M(\unbased{A})_{ij})$ associated to oriented corner arcs $\unbased{A}$ that were used in the construction of the quantum trace map.  In particular, $M(\unbased{A})$ is an $n\times n$ matrix with coefficients in the quantum torus $\qtorus{\ttriang}$.  This is a condensed version of \cite{douglasMR4717274} (or \cite{douglas2021points}).  

The \define{discrete triangle} is $\discrete=\{(a,b,c)\in\{0,1,2,\dots,n\}^3;a+b+c=n\}$.  We think of the discrete triangle overlayed on top of the ideal triangle $\ttriang$ such that the top (resp. bottom left or bottom right) vertex has coordinate $(0,n,0)$ (resp. $(n,0,0)$ or $(0,0,n)$).  See Figures \ref{fig:quiver} and \ref{fig:standard-arcs}.  A vertex $(a,b,c)$ of $\discrete$ is \define{interior} (resp. \define{corner} or \define{boundary}) if $a,b,c>0$ (resp. one of $a,b,c$ is $n$, or the vertex is not interior or corner).  Put $\cornerlessdiscrete=\discrete\setminus\{\text{corner vertices}\}$.  Define an antisymmetric function $P:\cornerlessdiscrete\times\cornerlessdiscrete\to\{-2,-1,0,1,+2\}$ by setting:  
\begin{equation*}
P(abc,a(b+1)(c-1))=P(abc,(a-1)b(c+1))=P(abc,(a+1)(b-1)c)=+2,
\end{equation*}
\begin{equation*}
P(abc,a(b-1)(c+1))=P(abc,(a+1)b(c-1))=P(abc,(a-1)(b+1)c)=-2
\end{equation*}
for $(a,b,c)$ interior, while for boundary vertices
\begin{equation*}
P(j(n-j)0,(j+1)(n-j-1)0)=P((n-j)0j,(n-j-1)0(j+1))=P(0j(n-j),0(j+1)(n-j-1))=+1
\end{equation*}
and, $0$ elsewhere.  Here, antisymmetric means $P(abc,a'b'c')=-P(a'b'c',abc)$.  Equivalently, one can read the function $P$ from the \define{quiver} with vertex set $\cornerlessdiscrete$ displayed in Figure \ref{fig:quiver}, where the solid (resp. dotted) arrows carry weight $+2$ (resp. $+1$).  Choose an arbitrary order of the $N=n(n+1)/2-3$ vertices of the quiver.  Then we can think of $P=(P_{ij})$ as an $N\times N$ antisymmetric matrix.

The quantum torus $\qtorus{\ttriang}$ is defined as the quotient of the free algebra in $2N$ indeterminates $X_1^{1/n},X_1^{-1/n},\dots,X_N^{1/n},X_N^{-1/n}$ by the relations 
\begin{equation*}
X_i^{m_i/n} X_j^{m_j/n}=\wroot^{m_i m_j P_{ij}} X_j^{m_j/n} X_i^{m_i/n},\,\, X_i^{m/n}X_i^{-m/n}=X_i^{-m/n}X_i^{m/n}=1 \,\, \lp m_i,m_j,m\in\Z\rp.
\end{equation*}
Here, $X_i^{m/n}$ is interpreted in the obvious way and we put $X_i=(X_i^{1/n})^n$.

Note that when $\wroot=1$, then $\cltorus{\ttriang}\cong\C[X_1^{\pm1/n},X_2^{\pm1/n},\dots,X_N^{\pm1/n}]$ is just the commutative Laurent polynomial algebra in the formal $n$-root variables $X_i^{1/n}$.  There is a linear, but not algebra, isomorphism $[-]:\cltorus{\ttriang}\to\qtorus{\ttriang}$, called the \define{Weyl quantum ordering}, defined on monomials by 
\begin{equation*}[X_1^{m_1/n}X_2^{m_2/n}\cdots X_N^{m_N/n}]=\wroot^{(-1/2)\sum_{1\leq i<j\leq N}m_i m_j P_{ij}}X_1^{m_1/n}X_2^{m_2/n}\cdots X_N^{m_N/n}.
\end{equation*}
Note that here is where the choice of the square root $\wroot^{1/2}$ sneaks into the construction.  

For $j=1,\dots,n-1$ and symbols $Z$ and $X$, define $n\times n$ diagonal and block diagonal~matrices 
\begin{align*}\elemedgemat{j}(Z)&=Z^{-j/n}\mrm{diag}\bmat Z&Z&\hdots&Z&1&1&\hdots&1\emat, \\
\elemleftmat{j}(X)&=X^{-(j-1)/n}\mrm{diag}\bmat X&\hdots&X&\bmat1&1\\0&1\emat&1&\hdots&1\emat,  \\ 
\elemrightmat{j}(X)&=X^{+(j-1)/n}\mrm{diag}\bmat1&\hdots&1&\bmat1&0\\1&1\emat&\inv{X}&\hdots&\inv{X}\emat
\end{align*}
 where $Z$ appears $j$ times and $X$ appears $j-1$ times, whenever such matrices make sense.  Note that $\elemleftmat{1}(X)=\elemleftmat{1}$ and $\elemrightmat{1}(X)=\elemrightmat{1}$ do not actually involve $X$.  

We will use the following notational conventions for products of matrices:  $\prod_{i=L}^N M_i=M_L M_{L+1}\cdots M_N$ and $\coprod_{i=N}^L M_i=M_N M_{N-1}\cdots M_L$ when $L\leq N$.  

Continuing, for symbols $Z_j$ and $X_{abc}$ for $j=1,2,\dots,n-1$ and $(a,b,c)$ interior vertices of $\discrete$, define matrices 
\begin{align*}
\edgemat(Z_1,\dots,Z_{n-1})&=\prod_{j=1}^{n-1}\elemedgemat{j}(Z_j),\\\leftmat((X_{abc})_{abc})&=\coprod_{i=n-1}^1\lp\elemleftmat{1}\prod_{j=2}^i\elemleftmat{j}(X_{(j-1)(n-i)(i-j+1)})\rp, \\
\rightmat((X_{abc})_{abc})&=\coprod_{i=n-1}^1\lp\elemrightmat{1}\prod_{j=2}^i\elemrightmat{j}(X_{(i-j+1)(n-i)(j-1)})\rp.
\end{align*}
Now, in the commutative polynomial algebra $\cltorus{\ttriang}$, put $Z_j=X_{j0(n-j)}$ and $Z^\prime_j=X_{j(n-j)0}$ and $Z^{\prime\prime}_j=X_{0j(n-j)}$ for $j=1,2,\dots,n-1$.  See Figure \ref{fig:standard-arcs}.  Define the \define{standard quantum left and right matrix} $\qleftmat$ and $\qrightmat$ by their entries
\begin{align*}
\qleftmat_{i,j}&=\left[\edgemat(Z_1,\dots,Z_{n-1})\leftmat((X_{abc})_{abc})\edgemat(Z^\prime_1,\dots,Z^\prime_{n-1})\right]_{\bar{i},\bar{j}},\\
\qrightmat_{i,j}&=\left[\edgemat(Z_1,\dots,Z_{n-1})\rightmat((X_{abc})_{abc})\edgemat(Z^{\prime\prime}_1,\dots,Z^{\prime\prime}_{n-1})\right]_{\bar{i},\bar{j}}.
\end{align*}
These are $n\times n$ matrices with coefficients in the quantum torus $\qtorus{\ttriang}$.  Note that, here, $\bar{i}=n+1-i$, and the Weyl quantum ordering of a matrix is defined in the obvious way by $[M]_{ij}=[M_{ij}]$.  So $\qleftmat$ is lower triangular and $\qrightmat$ is upper triangular.  

In Figure \ref{fig:standard-arcs}, we have displayed the \define{standard left and right arcs} $\unbased{A}_L$ and $\unbased{A}_R$.  Finally, define the desired quantum local monodromy matrices for the standard arcs by $M(\unbased{A}_L)=\qleftmat$ and $M(\unbased{A}_R)=\qrightmat$.  The matrices $M(\unbased{A})$ for the arcs $\unbased{A}$ on the other corners of the ideal triangle $\ttriang$ are defined by the triangular symmetry.  

\begin{figure}[htb!]
\centering
\includegraphics[scale=1.1]{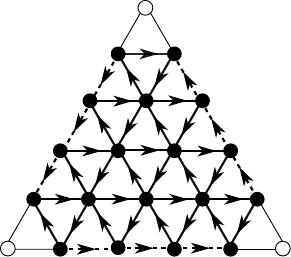}
\caption{Fock--Goncharov quiver.}
\label{fig:quiver}
\end{figure}

\begin{figure}[htb!]
\centering
\includegraphics[scale=.65]{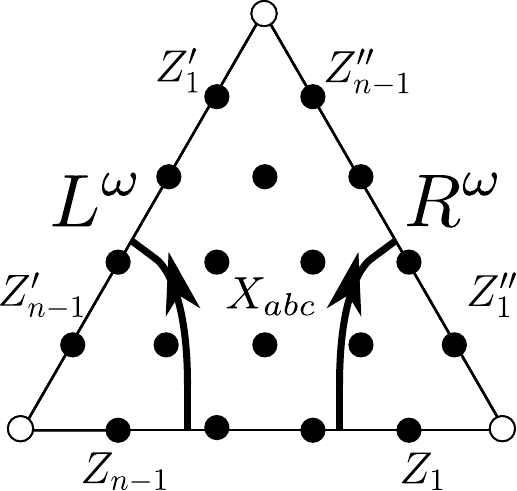}
\caption{Quantum left and right matrices.}
\label{fig:standard-arcs}
\end{figure}

Combined with Section \ref{sssec:definition-of-the-quantum-trace-map}, this completes the definition of the quantum trace map.  

\begin{remark}
The definition of the quantum trace map discussed in this and the previous sub-sub-section reproduces, essentially on the nose, the $n=n$ construction of \cite{le2023quantum}; see also \cite{douglas2021points, douglasMR4717274}.
\end{remark}

\subsubsection{Digression:  an open problem originating in the physical point of view}\label{sssec:quantum-traces-from-spectral-networks}

There is also a different construction of quantum traces from spectral networks \cite{gabellaMR3613514, neitzkeMR4190271, neitzkeMR4482975}, motivated by the physical problem of counting `framed BPS states'.  At least for simple oriented loops $\unbased{L}$, this different construction appears, experimentally, to agree with the construction of $\qtr(\unbased{L})$ given above.  For $n=2$, such an agreement has been put on a more solid mathematical footing \cite{kimMR4568006}.  

The physical interpretation of quantum traces leads to a precise mathematical prediction, essentially in \cite{gabellaMR3613514, neitzkeMR4190271, neitzkeMR4482975} and based on ideas in \cite{gaiottoMR3250763}, concerning their structure, and hints that even for simple closed curves $\unbased{L}$ the quantum trace $\qtr(\unbased{L})$ contains significantly more information than the classical trace.  The precise statement is that if we write
\begin{equation*}
\qtr(\unbased{L})=\sum_{\{i_1(\ttriang),\dots,i_N(\ttriang)\in\Z;\ttriang\in\idealtriang\}}
c(i_1(\ttriang),\dots,i_N(\ttriang);\ttriang\in\idealtriang)
\bigotimes_{\ttriang\in\idealtriang}\left[X_1^{i_1(\ttriang)/n}\cdots X_N^{i_N(\ttriang)/n}\right]
\end{equation*}
in $\bigotimes_{\ttriang\in\idealtriang}\qtorus{\ttriang}$, then the complex coefficients $c(i_1(\ttriang),i_2(\ttriang),\dots,i_N(\ttriang);\ttriang\in\idealtriang)$ with respect to the Weyl ordered monomials are sums of quantum positive integers
\begin{equation*}
c(i_1(\ttriang),i_2(\ttriang),\dots,i_N(\ttriang);\ttriang\in\idealtriang)
=\sum_{k>0} m_k[k]
\end{equation*}
where $[k]=(q^k-q^{-k})/(q-\inv{q})$ and $m_k\in\nnegZ$ is the multiplicity.  This is still an open problem, even for $n=2$.  Physically, these multiplicities are interpreted as counting the framed BPS states in quantum dimension $[k]$.  Some evidence is presented in \cite{neitzkeMR4482975}, including experiments performed together with the second author of the current paper.  

Note that if the curve $\unbased{L}$ satisfies the property that it passes through each triangle $\ttriang\in\idealtriang$ at most once, then it follows readily from the state sum formula that all the non-zero coefficients $c(i_1(\ttriang),i_2(\ttriang),\dots,i_N(\ttriang);\ttriang\in\idealtriang)$ are equal to $1*[1]=1$.  It is the `self interaction' that occurs when the curve passes through triangles more than once that allows for the creation of quantum coefficients $[k]$ for $k\geq2$.  Note also that these coefficients are global in nature, as they are the remnant of rewriting, for every triangle, the product over arcs of Weyl ordered terms as a scalar times a single Weyl ordered term, followed by collecting these local scalars into a single global scalar for the triangulation and gathering like monomials as a sum.  

As one final note, this prediction is true for simple loops in the annulus, with ideal triangulation as in Section \ref{sssec:annuluscomputationsetup} below, since a loop passes through each triangle exactly once.  Indeed, the only two simple loops are the clockwise and counterclockwise loop.  The recent paper \cite{kim2024quantized} could be helpful in studying this problem for general surfaces.  We emphasize that, at least a priori, this problem depends heavily on the choice of ideal triangulation.  

\begin{remark}
The prediction discussed in this sub-sub-section is related to the so-called `framed no-exotics conjecture' from the physical viewpoint \cite{gabellaMR3613514, gaiottoMR3250763, neitzkeMR4190271}.  Closely related mathematical conjectures appear in the context of quantum cluster varieties, see \cite{davisonMR3739230} and \cite[Conjecture 1.4]{davisonMR4337972}, where the analogous predicted property of the relevant coefficients is called `Lefschetz type'.  It is not entirely clear the degree to which these results and conjectures from quantum cluster geometry explain the above predicted property of quantum traces of loops.  For instance, there is a somewhat subtle distinction between the Fock--Goncharov cluster $\mcal{X}$ variables $X_{i}$ and their $n$-roots $X_i^{1/n}$; see \cite{kim2022rm, Kim_2024} for a detailed discussion of this distinction.  In particular, the above prediction applies to all simple loops, whose quantum traces are only guaranteed to be Laurent polynomials in the $X_i^{1/n}$, rather than in the $X_i$; contrast with the `quantum theta bases' of \cite{allegrettiMR3581328, mandel2023bracelets}.  
\end{remark}

\subsection{Annulus computation}\label{ssec:annulus-computation}

\subsubsection{Annulus:  setup}\label{sssec:annuluscomputationsetup}

In Sections \ref{sec:setting}-\ref{sec:spanning}, we worked in the open annulus $\ann$.  Since we now want to work with ideally triangulated punctured surfaces $\surf$, we redefine $\surf=\ann$ to be the punctured surface obtained as the closed annulus minus two points, one on each boundary component.  Certainly, this change has no effect on the skein algebras $\sk$ and $\usk$ studied in the above referenced sections.  The advantage is that in this new setting we can  incorporate the stated skein algebra $\ssk=\ssk(\ann)$ as well.  Note that inclusion induces a natural map $\incl:\usk\to\ssk$.  For $k=1,2,\dots,n-1$ recall that $\unbased{B}_k$ is the unbased web underlying the $k$-th irreducible based basis web $B_k$ appearing in Theorem \ref{thm:maintheorem}.  Put $\tilde{B}_k=\incl(\unbased{B}_k)$.  

In Figure \ref{fig:triangulated-annulus}, we show a good position of $\tilde{B}_k$ with respect to the split ideal triangulation $\splitidealtriang$ corresponding to the ideal triangulation $\idealtriang$ of the annulus $\ann$ with two biangles $\biang^L, \biang^R$ and two triangles $\ttriang^L, \ttriang^R$.  So what we usually think of as the counterclockwise direction corresponds to moving left in the rectangle in Figure \ref{fig:triangulated-annulus}.  

\begin{figure}[htb!]
\centering
\includegraphics[width=0.7\textwidth]{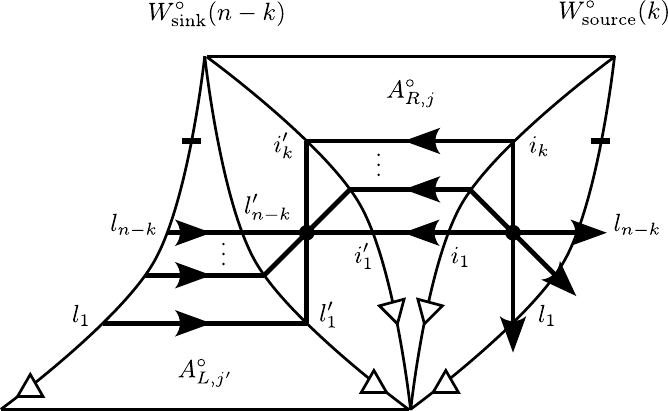}
\caption{Good position of the $k$-th irreducible basis web with respect to the split ideal triangulation of the annulus.}
\label{fig:triangulated-annulus}
\end{figure}

Note that the good position is such that the arcs $\unbased{A}_{L,1},\unbased{A}_{L,2}, \dots, \unbased{A}_{L,n-k}$ (resp. the arcs $\unbased{A}_{R,k}, \dots, \unbased{A}_{R,2}, \unbased{A}_{R,1}$) in the left triangle decrease (resp. increase) in height as they move toward the tip of the triangle that they straddle.  All these oriented arcs are left turning.  Note also that in the biangle on the left is the sink web $\sink(n-k)$ (Section \ref{ssec:punctured-surfaces}) and on the right is the source web $\source(k)$.  The calculations will be in terms of the compatibly stated restricted webs $\unbased{A}_{L,j^\prime}(l_{j^\prime},l^\prime_{j^\prime})$, $\sink(n-k,(l^\prime_{j^\prime})_{j^\prime},(i^\prime_j)_j)$, $\unbased{A}_{R,j}(i_j,i^\prime_j)$, $\source(k,(i_j)_j,(l_{j^\prime})_{j^\prime})$ in the restrictions $\ttriang^L\times\intt$, $\biang^L\times\intt$, $\ttriang^R\times\intt$, $\biang^R\times\intt$ respectively for $j=1,2,\dots,k$ and $j^\prime=1,2,\dots,n-k$ and where the states are as shown in Figure \ref{fig:triangulated-annulus}.  

Let the bottom right (resp. top left) vertex of the left (resp. right) triangle correspond to the top vertex of the discrete triangle $\discrete$ (Figure \ref{fig:standard-arcs}).  Let $X^{L(1/n)}_v=X^{L(1/n)}_{abc}$ (resp. $X^{R(1/n)}_v=X^{R(1/n)}_{abc}$) for $v=(a,b,c)\in\cornerlessdiscrete$ denote the generators in the quantum torus $\qtorus{\ttriang^L}$ (resp. $\qtorus{\ttriang^R}$) associated to the left (resp. right) triangle.  Our goal is to compute, at least in part, the quantum trace 

\begin{gather*}\label{eq:qtracecalculation}\tag{$\dagger$}
\qtr(\tilde{B}_k)=\sum_{(i_j)_j,(l_{j^\prime})_{j^\prime},(i^\prime_j)_j,(l^\prime_{j^\prime})_{j^\prime}}
\counit(\sink(n-k,(l^\prime_{j^\prime})_{j^\prime},(i^\prime_j)_j))\counit(\source(k,(i_j)_j,(l_{j^\prime})_{j^\prime}))\ast
\\\notag
\ast M(\unbased{A}_{L,1})(l_{1}l^\prime_{1})\cdots M(\unbased{A}_{L,n-k})(l_{n-k}l^\prime_{n-k})
\otimes M(\unbased{A}_{R,1})(i_1 i^\prime_1)\cdots M(\unbased{A}_{R,k})(i_k i^\prime_k) 
\end{gather*}
in $\qtorus{\ttriang^L}\otimes\qtorus{\ttriang^R}$ for each $k=1,2,\dots,n-1$.  

\subsubsection{Annulus:  a piecewise linear computation}\label{sssec:annuluscomputationsomemoregeneralities}

In general, for the quantum torus $\qtorus{\ttriang}$, we can rewrite any \define{monomial} $M=aX_{v_1}^{m_1/n}\cdots X_{v_k}^{m_k/n}=a\wroot^{d/2}[\prod_{v\in\cornerlessdiscrete}X_v^{d_v/n}]$ in $\qtorus{\ttriang}$ for $a\in\sC$ and $v_i\in\cornerlessdiscrete$ and $m_i,d,d_v\in\Z$ as a non-zero scalar times the Weyl ordered classical monomial in $\cltorus{\ttriang}$.  The \define{degree} $\degree(M)$ of this monomial is the function $v\mapsto \degree(v)=d_v\in\Z$ for $v\in\cornerlessdiscrete$.  Note that $\degree(M)$ is indeed well-defined on the class of $M$ in $\qtorus{\ttriang}$.  There is a partial order on degrees defined by $\degree(M)\geq\degree(M^\prime)$ if $\degree(M)(v)\geq\degree(M^\prime)(v)$ for all $v\in\cornerlessdiscrete$.  A monomial $M$ appearing as a summand of an element $P$ of $\qtorus{\ttriang}$ is a \define{highest degree monomial} of $P$ if $\degree(M)>\degree(M^\prime)$ for all other monomials $M^\prime$ appearing as summands of the Laurent polynomial $P$.  

\begin{lemma}\label{lem:maincalculationallemma}
Let $\qleftmat$ and $\qrightmat$ be the standard quantum left and right matrices.  In general, an entry $\qleftmat_{ij}$ or $\qrightmat_{ij}$ is a sum of monomials in $\qtorus{\ttriang}$.  Put $M=\qleftmat$, which recall is lower triangular.  Note that similar facts hold for the right matrix as well, which we will not use.
\begin{enumerate}
\item  $M_{ii}$ is a monomial.
\item  $M_{ii}$ commutes with $M_{i^\prime i^\prime}$.  
\item  $\degree(M_{ii})>\degree(M_{(i-1)(i-1)})$, and $\degree{M}_{ii}$ is strictly greater than any of the degrees of the monomials appearing as summands in $M_{ij}$ for $j<i$. 
\item  For all $(a,b,c)\in\cornerlessdiscrete$, $\degree(M_{ii})(a,b,c)=\begin{cases}n-a,&n+1-i\leq a,1\leq a\leq n-1;\\-a,&n+1-i>a,1\leq a\leq n-1;\\0,&a=0.\end{cases}$  
\item  For $k=1,2,\dots,n-1$, define functions $f_k:\R\to\R$ and $g_k:\R\to\R$ by $f_k(x)=kx$ and $g_k(x)=n(n-k)-(n-k)x$.  Note that $f_k(0)=g_k(n)=0$ and $f_k(n-k)=g_k(n-k)$.  Then for all $(a,b,c)\in\cornerlessdiscrete$
\begin{align*}
\degree\lp\prod_{i=n+1-k}^n M_{ii}\rp(a,b,c)&=\min\{f_k(n-a),g_k(n-a)\}\in\nnegZ,\\
\degree\lp\prod_{i=k+1}^n M_{ii}\rp(a,b,c)&=\min\{f_k(a),g_k(a)\}\in\nnegZ.
\end{align*}
\end{enumerate}
\end{lemma}

\begin{proof}
(1)  By definition, $M$ is the product of lower triangular, and diagonal, matrices $M_v$, $v\in\cornerlessdiscrete$, each of whose $ij$-entry is of the form $X_v^{d_{vij}/n}$ for some $d_{vij}\in\Z$, as well as lower triangular matrices involving no variables.  Rather, $M$ is the Weyl ordering of such a product, but the Weyl ordering will play no role throughout this proof since we are only interested in degrees, so we will omit it for notational simplicity.  

(2)  This is one of the relations appearing in the main result of \cite{douglas2021points}. 

(3)  The first statement is immediate from the form of the matrices $M_v$.  The second statement is because the off diagonal terms $X_v^{d_{vij}/n}$ of the $M_v$ have exponent $d_{vij}<0$.

(4)  Clearly $M$ does not involve the $X_{0j(n-j)}$, so assume $a\geq1$.  Recall that the $(\bar{m},\bar{m})$ entry of $M$ is the $(m,m)$ entry of the product of matrices $\elemedgemat{l}(X_{l0(n-l)})$ and $\elemedgemat{l}(X_{l(n-l)0})$ for $1\leq l\leq n-1$ and $\elemleftmat{j}(X_{(j-1)(n-i)(i-j+1)})$ for $2\leq j\leq i\leq n-1$, and $\elemleftmat{1}$ a number of times.  We only need to observe then that, in the former, $X_{l0(n-l)}$ and $X_{l(n-l)0}$ appear with power $1-\ell/n$ (resp. $-\ell/n$) when $m\leq l$ (resp. $m>l$), and, in the latter, $X_{(j-1)(n-i)(i-j+1)}$ appears with power $1-(j-1)/n$ (resp. $-(j-1)/n$) when $m\leq j-1$ (resp. $m>j-1$).  

(5)  This is part of the mathematical content of \cite[Equation (172) and Theorem 10.2]{le2023quantum}; for an earlier appearance of these piecewise linear formulas, see \cite[Theorem 8.22]{SunGeomFunctAnal20}.  For the sake of self-containedness, we provide this short calculation here as well.   The formulas work when $a=0$, so assume $a\geq1$.  Note that the larger degree $n-a$ for $M_{ii}$ occurs if and only if $n+1-a\leq i$, else it has the smaller degree $-a$.  

We begin with the first equality, summing the degrees of the $M_{ii}$ for $n+1-k\leq i\leq n$.  Note that $f_k(n-a)\leq g_k(n-a)$ if and only if $n-a\leq n-k$.  There are two cases.  If $k\leq a$, then $n+1-a\leq n+1-k\leq i$, so the sum is $n-(n+1-k)+1$ terms of weight $n-a$, which is $k(n-a)=f_k(n-a)$ as desired.  If $k>a$, then the sum is $(n+1-a)-(n+1-k)$ terms of weight $-a$ plus $n-(n+1-a)+1$ terms of weight $n-a$, which is $(k-a)(-a)+a(n-a)=(n-k)a=g_k(n-a)$ as desired.  

The second equality sums the degrees of $M_{ii}$ for $k+1\leq i\leq n$.  Note that $f_k(a)\leq g_k(a)$ if and only if $a\leq n-k$.  There are two cases.  If $a>n-k$, then $n+1-a<k+1\leq i$, so the sum is $n-(k+1)+1$ terms of weight $n-a$, which is $(n-k)(n-a)=g_k(a)$ as desired.  If $a\leq n-k$, then the sum is $(n+1-a)-(k+1)$ terms of weight $-a$ plus $n-(n+1-a)+1$ terms of weight $n-a$, which is $(n-a-k)(-a)+a(n-a)=ka=f_k(a)$ as desired.  
\end{proof}

\begin{remark}
It has been known by experts for some time now, see for example the Acknowledgements and the comment about \cite{SunGeomFunctAnal20} in the proof of item (5) above as well as the upcoming work \cite{leshensunweng}, that a certain set of equivalence classes of $n$-valent graphs in a punctured surface $\surf$ equipped with an ideal triangulation $\idealtriang$ is in natural one-to-one correspondence with a certain subset of $\nnegZ^N$ where $N$ depends only on $n$ and the topology of $\surf$; see also \cite{xie2013higherlaminationswebsn2}.  Here, the graph is called a `reduced web' and the corresponding $N$-tuple constitutes its `tropical web coordinates'.  Such a result generalizes to the case $n=n$ the work \cite{FockArxiv97, fockMR2233852} in the case $n=2$ and the work \cite{douglasMR4685684, douglas2022tropical, Frohman22MathZ} in the case $n=3$.  

When the surface $\surf=\ann$ is the annulus, the reduced webs represent the same skein basis webs being studied in the present paper.  Although, technically, we do not explicitly use this $n=n$ result about tropical web coordinates, it essentially underlies our entire strategy.  This is because the $N$-tuple of tropical web coordinates is precisely the $N$-tuple of exponents of the highest term of the quantum trace polynomial of the basis web, which we compute in the current section.  See \cite{bonahonMR2851072, kim2022rm} for the corresponding quantum trace calculations for $n=2$ and $n=3$ for any surface $\surf$.  

Lastly, building on what was learned from the $n=2$ and $n=3$ cases, we owe the majority of our understanding of Fock--Goncharov duality in the $n=n$ case, discussed in Section \ref{sec:FGduality}, to conversations with Zhe Sun; see the Acknowledgements.
\end{remark}

\subsubsection{Annulus:  finishing}\label{sssec:annuluscomputationfinishing}
Returning from generalities to the problem at hand, for $k=1,2,\dots,n-1$ define the \define{$k$-th irreducible tropical coordinate functions} $t^L_k:\cornerlessdiscrete\to\nnegZ$ and $t^R_k:\cornerlessdiscrete\to\nnegZ$ by $t^L_k(a,b,c)=\min\{f_k(a),g_k(a)\}$ and $t^R_k(a,b,c)=\min\{f_k(n-a),g_k(n-a)\}$ where the functions $f_k:\R\to\R$ and $g_k:\R\to\R$ are defined in Part (5) of Lemma \ref{lem:maincalculationallemma}.  Recall that here we are thinking of $\cornerlessdiscrete$ overlayed on top of the triangles $\ttriang^L$ and $\ttriang^R$ of the split ideal triangulation $\splitidealtriang$ of the annulus $\ann$ as prescribed above in Section \ref{sssec:annuluscomputationsetup}.  Note that $t^L_k(a,b,c)=t^R_k(a,b,c)=0$ for $a=0$.  Although we will not use them explicitly in this sub-sub-section, note that we have the following symmetries:  $t^L_k(n-a,b,c)=t^R_k(a,b,c)$ and $t^L_{n-k}(a,b,c)=t^R_k(a,b,c)$.  We will return to these symmetries in Section \ref{sec:FGduality} below.  

Write the quantum trace \eqref{eq:qtracecalculation} for $\tilde{B}_k$ in the form $\qtr(\tilde{B}_k)=\sum_i a_i M^L_i\otimes M^R_i$ for monomials $M^L_i$ and $M^R_i$ in $\qtorus{\ttriang^L}$ and $\qtorus{\ttriang^R}$ respectively, collecting terms so that no two summands have the same pair of monomials.  Let $\badset\subset\sC$ be defined as in Theorem {\upshape\ref{thm:maintheorem}}.  We claim there is a single $i=h$, independent of $\qroot\in\sC\setminus\badset$, such that $M^L_h\otimes M^R_h$ is a highest degree term, in the sense that $\degree(M_h^L)>\degree(M_i^L)$ and $\degree(M_h^R)>\degree(M_i^R)$ for all $i$.  More precisely, we claim that $\degree(M_h^L)=t_k^L$ and $\degree(M_h^R)=t_k^R$ as functions $\cornerlessdiscrete\to\nnegZ$.  

Labeling states as in Figure \ref{fig:triangulated-annulus}, by Part (3) of Lemma \ref{lem:maincalculationallemma}, together with the formula in Figure \ref{fig:counit-example}, we see that the highest degree pair $M_h^L\otimes M_h^R$ of monomials occurs for the states $l_{j^\prime}=l^\prime_{j^\prime}=k+j^\prime$ for $j^\prime=1,2,\dots,n-k$ and $i_j=i^\prime_j=n+j-k$ for $j=1,2,\dots,k$, and their permutations:  so, more precisely, for the states $l_{\sigma^\prime(j^\prime)}=l^\prime_{\sigma^\prime(j^\prime)}=k+j^\prime$ for $j^\prime=1,2,\dots,n-k$ and $i_{\sigma(j)}=i^\prime_{\sigma(j)}=n+j-k$ for $j=1,2,\dots,k$, varying over $\sigma^\prime\in\sym_{n-k}$ and $\sigma\in\sym_k$.  Indeed, this will be true so long as the coefficient resulting from collecting terms is non-zero.  By Lemma \ref{lem:summingcounitoverthesymmetricgroup} and Part (2) of Lemma \ref{lem:maincalculationallemma}, for $(i_j)_j$, $(l_{j^\prime})_{j^\prime}$ as above, and putting $M=\qleftmat$ to be the quantum left matrix, we compute this highest term in $\qtorus{\ttriang^L}\otimes\qtorus{\ttriang^R}$ of \eqref{eq:qtracecalculation} as
\begin{gather*}
\sum_{\sigma^\prime\in\sym_{n-k},\sigma\in\sym_{k}}
\counit(\sink(n-k,(l_{\sigma^\prime(j^\prime)})_{j^\prime},(i_{\sigma(j)})_j))\counit(\source(k,(i_{\sigma(j)})_j,(l_{\sigma^\prime(j^\prime)})_{j^\prime}))\ast
\\
\ast M(\unbased{A}_{L,1})(l_{\sigma^\prime(1)}l_{\sigma^\prime(1)})\cdots M(\unbased{A}_{L,n-k})(l_{\sigma^\prime(n-k)}l_{\sigma^\prime(n-k)})
\otimes M(\unbased{A}_{R,1})(i_{\sigma(1)}i_{\sigma(1)})\cdots M(\unbased{A}_{R,k})(i_{\sigma(k)}i_{\sigma(k)}) 
\\=q^{(n-k)(n-k-1)/2}[n-k]!q^{k(k-1)/2}[k]!
\counit(\sink(n-k,(l_{j^\prime})_{j^\prime},(i_{j})_j))\counit(\source(k,(i_{j})_j,(l_{j^\prime})_{j^\prime}))\ast
\\\ast M_{(k+1)(k+1)}\cdots M_{(n-1)(n-1)}M_{nn}\otimes M_{(n+1-k)(n+1-k)}\cdots M_{(n-1)(n-1)}M_{nn}
\end{gather*}
which is non-zero for $\qroot\in\sC\setminus\badset$, as desired.  And we also see, by Part (5) of Lemma \ref{lem:maincalculationallemma}, that $\degree(M_h^L)=t_k^L$ and $\degree(M_h^R)=t_k^R$ as claimed.  

\begin{lemma}\label{lem:tropicalindependence}
Thought of as elements of the vector space $\R^{\cornerlessdiscrete}$, the $n-1$ functions $t^R_k$ are linearly independent.  Similarly for the functions $t^L_k$.  
\end{lemma}

\begin{proof}
Abusing notation and putting $t^R_k:\R\to\R$ defined by $t^R_k(a)=\min\{f_k(n-a),g_k(n-a)\}$, it suffices to show that the $n-1$ vectors $(t^R_k(n-1),\dots,t^R_k(2),t^R_k(1))$ in $\R^{n-1}$ varying over $k=1,2,\dots,n-1$ are linearly independent.  Indeed, if we arrange these $n-1$ vectors as the columns of a $(n-1)\times(n-1)$ matrix $T^R$, $T^R_{ij}=t^R_j(n-i)$, then this matrix is the inverse of the matrix $\inv{\lp T^R\rp}=(1/n)\bmat&&&&0&0&-1&2\\&&&&0&-1&2&-1\\&&&&-1&2&-1&0\\\vdots&\vdots&\vdots&\reflectbox{$\ddots$}&&&&\\-1&2&-1&\hdots&&&&\\2&-1&0&\hdots&&&&\emat$.  \end{proof}
\begin{remark}
These  matrices appear in the representation theory of  $\nsl$.  They will make another appearance in Section \ref{sec:FGduality} below.  
\end{remark}

\begin{proposition}\label{prop:second-main-result-independenceforstatedskeins}
Let $\badset\subset\sC$ be defined as in Theorem {\upshape\ref{thm:maintheorem}}.  For all $\qroot\in\sC\setminus\badset$, the monomials of the form $\tilde{B}_{1}^{m_1}\tilde{B}_{2}^{m_2}\cdots \tilde{B}_{n-1}^{m_{n-1}}\in\ssk$ for $m_i\in\nnegZ$ are linearly independent in the stated skein algebra $\ssk$ of the annulus.
\end{proposition}

\begin{proof}
By the computation performed above, and since the quantum trace map $\qtr$ is an algebra homomorphism, we have already proved that the highest degree term of the quantum trace $\qtr(\tilde{B}_{1}^{m_1}\tilde{B}_{2}^{m_2}\cdots \tilde{B}_{n-1}^{m_{n-1}})$ has left degree $\sum_{k=1}^{n-1} m_k t^L_k$ with respect to $\qtorus{\ttriang^L}$ and right degree $\sum_{k=1}^{n-1} m_k t^R_k$ with respect to $\qtorus{\ttriang^R}$, these degrees which we remind are functions $\cornerlessdiscrete\to\nnegZ$.  In fact, we only need to consider, say, the $n-1$ right degree values on points $(n-a,0,a)\in\cornerlessdiscrete$ for $a=1,2,\dots,n-1$, as in the proof of Lemma \ref{lem:tropicalindependence}.  Suppose there were a non-trivial linear dependence relation among this family of monomial webs.  By Lemma \ref{lem:tropicalindependence}, the highest degrees of their quantum traces are distinct, so considering a monomial $\tilde{B}_{1}^{m_1}\tilde{B}_{2}^{m_2}\cdots \tilde{B}_{n-1}^{m_{n-1}}$ whose quantum trace has maximal degree with respect to the partial order on degrees, the coefficient of that monomial in the dependence relation must be zero, which is a contradiction.  
\end{proof}

\begin{proof}[Proof of Proposition {\upshape\ref{prop:second-main-result-independence}}]
First, by definition the inclusion linear map $\incl:\usk\to\ssk$ sends the unbased monomial webs $(\unbased{B}_{1})^{m_1}(\unbased{B}_{2})^{m_2}\cdots(\unbased{B}_{n-1})^{m_{n-1}}$ to the monomial webs $\tilde{B}_{1}^{m_1}\tilde{B}_{2}^{m_2}\cdots \tilde{B}_{n-1}^{m_{n-1}}$.  Since the latter are linearly independent, so are the former.  Then, the result for the based monomial webs $B_{1}^{m_1}B_{2}^{m_2}\cdots B_{n-1}^{m_{n-1}}$ follows by Proposition \ref{prop:spinstructures}.  
\end{proof}

\begin{proof}[Proof of Theorem {\upshape\ref{thm:maintheorem}}]
This is the combination of Propositions \ref{lem:lastsentenceofmaintheorem}, \ref{cor:basisgenerates}, and \ref{prop:second-main-result-independence}.
\end{proof}

\section{Applications}\label{sec:applications}

\subsection{Embedding the commutative polynomial algebra into quantum tori}\label{ssec:injectivity}

By Proposition \ref{prop:spinstructures} together with the proof of Theorem \ref{thm:maintheorem}, we immediately obtain:

\begin{corollary}\label{cor:firstapplication}
For the ideal triangulation $\idealtriang$ of the annulus $\ann$ as in Section {\upshape\ref{sssec:annuluscomputationsetup}} and for $\qroot\in\sC\setminus\badset$ the composition 
\begin{equation*}
\qtr\circ\incl:\sk\cong\usk\to\ssk\to\qtorus{\ttriang^L}\otimes\qtorus{\ttriang^R}
\end{equation*}
is an injective algebra homomorphism.  
\qed
\end{corollary}

In other words, for every such $\qroot$ we obtain an explicit geometric embedding of the polynomial algebra $\sk\cong\C[x_1,x_2,\dots,x_{n-1}]$ into the quantum torus $\qtorus{\ttriang^L}\otimes\qtorus{\ttriang^R}$.

\begin{remark}
Given any quiver, the Weyl quantum ordering defines a linear isomorphism from the polynomial algebra generated by the vertices of the quiver to a quantum torus on the same number of generators, but this is an algebra isomorphism if and only if the quiver is trivial, or $\qroot=1$.  Compare to Corollary \ref{cor:firstapplication}, which embeds the polynomial algebra into a quantum torus on many more generators.  

As another observation, Corollary \ref{cor:firstapplication} implies that the highest degree monomial terms of the quantum traces of the monomial web basis elements commute in $\qtorus{\ttriang^L}\otimes\qtorus{\ttriang^R}$.  
\end{remark}

In the cases $n=2$ and $n=3$ \cite{bonahonMR2851072, douglasMR4685684, kim2022rm}, where the analogue of Corollary \ref{cor:firstapplication} is true for general surfaces $\surf$, this analogue gives a proof that the skein algebra has no non-trivial zero divisors.  Since polynomial algebras clearly satisfy this latter property, by Theorem \ref{thm:maintheorem} we immediately obtain:

\begin{corollary}\label{cor:nonilpotents}
For $\qroot\in\sC\setminus\badset$, the skein algebra $\sk\cong\C[x_1,x_2,\dots,x_{n-1}]$ of the annulus has no non-trivial zero divisors, in particular, has no non-trivial nilpotent elements.  \qed
\end{corollary}

By \cite[Corollary 20]{SikoraAlgGeomTop05}, see also \cite{bullock1997rings, przytycki1997skein} for the case $n=2$, the skein module $\sk(\man)$ of a three-dimensional manifold $\man$ (Remark \ref{rem:surfaces}) evaluated at $\qroot=1$ is isomorphic to the, a priori, unreduced coordinate ring over $\C$ of the $\nsl$ character variety of $\man$.  Said another way, the reduced coordinate ring of the $\nsl$ character variety of $\man$ is isomorphic to $\sk(\man)(\qroot=1)/\sqrt{0}$, namely the quotient of $\sk(\man)(\qroot=1)$ by its set of nilpotent elements.  By Corollary \ref{cor:nonilpotents}, since $\qroot=1\notin\badset$, taking $\man=\ann\times\intt$, the coordinate ring of the $\nsl$ character variety of the annulus is reduced and isomorphic to the skein algebra $\sk(\qroot=1)\cong\C[x_1,x_2,\dots,x_{n-1}]$.   

\subsection{Injectivity of the splitting map for the twice punctured sphere}\label{ssec:injectivityofthesplittingmap}

Let two surfaces $\surf$ and $\surf_\alpha$ be related by cutting along an ideal arc $\alpha$, as described at the end of Section \ref{split-ideal-triangulations-good-position-and-definition-of-the-quantum-trace-map}, and let $\ssk(\surf)$ and $\ssk(\surf_\alpha)$ be the associated stated skein algebras.  In \cite{leMR3827810, LeArxiv21} there is defined a natural algebra homomorphism $\ssk(\surf)\to\ssk(\surf_\alpha)$ called the splitting map.  It is defined as follows.  Let $\unbased{W}$ be a web in $\surf\times\intt$.  Perturb $\unbased{W}$ to be in a generic position with respect to the wall $\alpha\times\intt$, the most important property of which is that the points $\unbased{W}\cap(\alpha\times\intt)$ lie at distinct heights.  Let $\unbased{W}_\alpha$ be the resulting, not naturally stated, split web in $\surf_\alpha\times\intt$.  The image of $\unbased{W}$ in $\ssk(\surf_\alpha)$ under the splitting map is the sum of the compatibly stated split webs $\unbased{W}_\alpha(s_\alpha)$ varying over all compatible states $s_\alpha$ (Section \ref{split-ideal-triangulations-good-position-and-definition-of-the-quantum-trace-map}).  In particular, this image is independent of the perturbation, which could possibly swap heights thereby changing the isotopy class of the split web $W_\alpha$ in $\surf_\alpha\times\intt$.  

If $\alpha$ is a short ideal arc in $\surf$ starting and ending at the same puncture, so that the split surface $\surf_\alpha=\bar{\surf}\cup\monoangle$ is obtained from $\surf$ by cutting out an ideal monoangle $\monoangle$, yielding a surface $\bar{\surf}$ with an extra boundary component with one puncture, then by Theorem \ref{thm:monoanglebiangletriangle} we have the identifications $\ssk(\surf_\alpha)=\ssk(\bar{\surf}\cup\monoangle)=\ssk(\bar{\surf})\otimes\ssk(\monoangle)=\ssk(\bar{\surf})\otimes\C=\ssk(\bar{\surf})$.  

Note that yet another surface $\surf=\ann$ yielding the same skein algebra $\usk$ we have been studying throughout the paper is the twice punctured sphere.  In this case, $\surf$ has empty boundary, so we have $\usk=\ssk$.  Here is the second application of the main result.  

\begin{corollary}\label{cor:secondapplication}
If $\surf=\ann$ is the twice punctured sphere, and $\alpha$ is a short ideal arc in $\surf$ starting and ending at the same puncture, then for $\qroot\in\sC\setminus\badset$ the splitting map 
\begin{equation*}
\sk\cong\usk=\ssk\to\ssk(\surf_\alpha)=\ssk(\bar{\surf})
\end{equation*} 
of \cite{LeArxiv21} is injective.  Here $\bar{\surf}$ is the sphere with one internal puncture and one boundary component with a single puncture.  
\end{corollary}

\begin{proof}
 This is immediate once it is known that the monomial webs $(\unbased{B}_{1})^{\prime m_1}(\unbased{B}_{2})^{\prime m_2}\cdots (\unbased{B}_{n-1})^{\prime m_{n-1}}$ in the stated skein algebra $\ssk(\bar{\surf})$, which are the images under the splitting map of the monomial basis webs $(\unbased{B}_{1})^{m_1}(\unbased{B}_{2})^{m_2}\cdots (\unbased{B}_{n-1})^{m_{n-1}}$ in $\usk$, are linearly independent in  $\ssk(\bar{\surf})$.  This was essentially proved above as part of the proof of the main theorem.  Technically, this independence was proved for the annular surface $\bar{\bar{\surf}}$ obtained from $\bar{\surf}$ by cutting out another monoangle from the remaining internal puncture, but since the monomial webs in the surface $\bar{\bar{\surf}}$ are the images under the splitting map of the monomial webs in the surface $\bar{\surf}$, the latter are independent as well.    Compare the $n=3$ result \cite[Theorem 8.1]{higginsMR4609753}, or the earlier $n=2$ result of \cite{leMR3827810}, both proved for any surface $\surf$.
\end{proof}

\section{Fock--Goncharov duality}\label{sec:FGduality}

Given a function $f:\cornerlessdiscrete\to\R$, written $f\in\R^{\cornerlessdiscrete}$, for $i=1,2,\dots,n-1$ and $j=1,2,\dots,n-i$ the associated \define{top} (resp. \define{bottom left} or \define{bottom right}) \define{rhombus number} $r(f)^\mrm{t}_{ij}$ (resp. $r(f)^\mrm{bl}_{ij}$ or $r(f)^\mrm{br}_{ij}$) in $\R$ is defined by
\begin{equation*}
r(f)^\mrm{t}_{ij}=(f(n-i-j+1,i,j-1)+f(n-i-j,i,j)-f(n-i-j+1,i-1,j)-f(n-i-j,i+1,j-1))/n,
\end{equation*}
\begin{equation*}
r(f)^\mrm{bl}_{ij}=(f(i,j-1,n-i-j+1)+f(i,j,n-i-j)-f(i-1,j,n-i-j+1)-f(i+1,j-1,n-i-j))/n,
\end{equation*}
\begin{equation*}
r(f)^\mrm{br}_{ij}=(f(j-1,n-i-j+1,i)+f(j,n-i-j,i)-f(j,n-i-j+1,i-1)-f(j-1,n-i-j,i+1))/n.
\end{equation*}
Note that the four vertices entering the expression for each rhombus number are the vertices of a small rhombus in the discrete triangle $\discrete$, where the obtuse (resp. acute) vertices are added (resp. subtracted).  Only the top (resp. bottom left or bottom right) rhombus $r(f)^\mrm{t}_{(n-1)1}$ (resp. $r(f)^\mrm{bl}_{(n-1)1}$ or $r(f)^\mrm{br}_{(n-1)1}$) includes the acute corner vertex $(0,n,0)$ (resp. $(n,0,0)$ or $(0,0,n)$), where above by definition we put $f(v)=0$ on these corner vertices $v$.  

The \define{balanced Knutson--Tao lattice} (resp. \define{fan}) is the set of functions $f:\cornerlessdiscrete\to\R$ satisfying the property that all the rhombus numbers $r(f)^\mrm{t}_{ij}$ and $r(f)^\mrm{bl}_{ij}$ and $r(f)^\mrm{br}_{ij}$ are integers (resp. natural numbers)  \cite{KnutsonJAmerMathsoc99}.  Note that the balanced Knutson--Tao lattice (resp. fan) is a subgroup (resp. sub-monoid) of $\R^{\cornerlessdiscrete}$ under addition, in particular it contains the zero function.  We denote by $\fan$ the sub-monoid that is the intersection of the balanced Knutson--Tao fan with those functions $f$ satisfying $f(a,b,c)=0$ for all $a=0$.  

Let the functions $t^L_k$ and $t^R_k$ in $\nnegZ^{\cornerlessdiscrete}$ for $k=1,2,\dots,n-1$ be as in Section \ref{sssec:annuluscomputationfinishing}.  One checks $t^R_k\in\fan$ for each $k$.  By the symmetry $t^L_{n-k}=t^R_{k}$ this is also true for the $t^L_k$.  Moreover, we have the following specific values for the rhombus numbers:  $r(t^R_k)^\mrm{bl}_{kj}=1$ for all $j=1,2,\dots,n-k$ and all other rhombus numbers $r(t^R_k)^\mrm{bl}_{ij}$, $i\neq k$, and $r(t^R_k)^\mrm{t}_{ij}$ and $r(t^R_k)^\mrm{br}_{ij}$ are zero. 

A non-zero element $f$ of a monoid is \define{irreducible} if $f$ cannot be written as the sum of two non-zero elements of the monoid.  If it happens that this monoid is \define{positive}, namely a sub-monoid of $\nnegZ^m$, then it is not hard to see that the set of irreducible elements forms the unique minimum spanning set over $\nnegZ$ of the monoid, which we call the \define{Hilbert basis}; see, for instance, \cite[Section 3]{douglas2022tropical}.  

\begin{proposition}\label{thm:KTfan}
The sub-monoid $\fan$ of the balanced Knutson--Tao fan defined above is positive, namely $\fan\subset\nnegZ^{\cornerlessdiscrete}$, and its Hilbert basis consists of the $n-1$ elements $t^R_k$ for $k=1,2,\dots,n-1$.  Viewing $\fan\subset\nnegZ^{n-1}$ by the inclusion $f\mapsto(f(n-1,b,c), \dots, f(2,b',c'), f(1,b'',c''))$, which is independent of the choices of $b,c,\dots,b',c',b'',c''$, these $n-1$ Hilbert basis elements are linearly independent over $\R$, so span over $\R_{\geq0}$ a single real $(n-1)$-dimensional cone in $\R_{\geq0}^{n-1}$ containing $\fan$ as a sub-monoid.  
\end{proposition}

\begin{proof}
If we can show that $\fan$ is contained in the span over $\nnegZ$ of $\{t_k^R;k=1,2,\dots,n-1\}\subset\nnegZ^{\cornerlessdiscrete}$, this will establish that $\fan$ is positive.  That these $n-1$ elements form the Hilbert basis of $\fan$, and the second statement of the proposition, in particular that the proposed inclusion $\fan\to\nnegZ^{n-1}$ is indeed injective, then follows by the linear independence property of Lemma~\ref{lem:tropicalindependence}.  

So assume $f\in\fan\subset\R^{\cornerlessdiscrete}$.  We want to produce $c_k\in\nnegZ$ such that $f=\sum_k c_k t^R_k$.  By definition of $\fan$, the rhombus numbers $r(f)^\mrm{bl}$, $r(f)^\mrm{t}$, $r(f)^\mrm{br}$ are in $\nnegZ$, and $f(a,b,c)=0$ for $a=0$.  We claim that $r(f)^\mrm{bl}_{ij}$ is independent of $j$ and that we can take $c_k=r(f)^\mrm{bl}_{kj}$.  

We begin by considering the rhombus numbers $r(f)^\mrm{t}_{i(n-i)}$ and $r(f)^\mrm{br}_{i1}$, which hug the edge of the triangle between the top and bottom right vertices.  Each such rhombus has an acute and obtuse vertex lying on this edge with the same value under $f$, namely zero in this case, and each adjacent pair $(1,i-1,n-i),(1,i,n-i-1)$ of vertices is a common side of a unique top and bottom right rhombus, namely $r(f)^\mrm{t}_{i(n-i)}$ and $r(f)^\mrm{br}_{(n-i)1}$.  These rhombus numbers are $\pm (f(1,i-1,n-i)-f(1,i,n-i-1))/n$.  Since they are both in $\nnegZ$ by assumption, they are zero, so we have $f(1,i-1,n-i)=f(1,i,n-i-1)$ for all $i$.  That is, $f(1,b,c)$ is constant, independent of $b,c$.  By induction on $a$, the same argument implies $f(a,b,c)$ is independent of $b,c$.  Abusing notation, write $f(a)=f(a,b,c)$.  It follows that the rhombus number $r(f)^\mrm{bl}_{kj}=(-f({k-1})+2f(k)-f({k+1}))/n$ is independent of $j$, so call it $c_k$.  This completes the first part of the claim.  

We gather we have a linear system $\bmat c_{n-1}\\\vdots\\c_1\emat=T^{-1}\bmat f(1)\\\vdots\\f({n-1})\emat$ where the matrix $T^{-1}$ appeared in the proof of Lemma \ref{lem:tropicalindependence}.  Abusing notation and putting $t^R_k:\R\to\R$ and $t^L_k:\R\to\R$ as in the proof of the same lemma, solving this system for the values of $f$, 
\begin{align*}
\bmat f(1)\\\vdots\\f({n-1})\emat
&=T\bmat c_{n-1}\\\vdots\\c_1\emat
=\sum_k c_{n-k}\bmat t^R_k(n-1)\\\vdots\\t^R_k(1)\emat
=\sum_k c_{n-k}\bmat t^L_k(1)\\\vdots\\t^L_k(n-1)\emat
\\&=\sum_k c_{n-k}\bmat t^R_{n-k}(1)\\\vdots\\t^R_{n-k}(n-1)\emat
=\sum_k c_k\bmat t^R_{k}(1)\\\vdots\\t^R_{k}(n-1)\emat
\end{align*}
where the third and fourth equalities use the symmetries discussed at the beginning of Section \ref{sssec:annuluscomputationfinishing}.  In other words, $f=\sum_k c_k t^R_k$, completing the second part of the claim.    
\end{proof}

We have so far been discussing the positive sub-monoid $\fan\subset\nnegZ^{\cornerlessdiscrete}$ associated to the discrete triangle.  Consider now the triangulated annulus $\ann$ as in Section {\upshape\ref{sssec:annuluscomputationsetup}}.  Associate a copy $\fan^L$ (resp. $\fan^R$) of $\fan$ to the left (resp. right) triangle $\ttriang^L$ (resp. $\ttriang^R$) of the ideal triangulation $\idealtriang$ of $\ann$.  The \define{balanced Knutson--Tao fan for the annulus}, denoted $\fan(\ann,\idealtriang)$, is the sub-set of elements $(f^L,f^R)$ of the Cartesian product $\fan\times\fan$ satisfying $f^L(a,0,n-a)=f^R(n-a,0,a)$ and $f^L(a,n-a,0)=f^R(n-a,a,0)$ for all $a$.  

For example, the pair $(t_k^L,t_k^R)$ is in $\fan(\ann,\idealtriang)$ for all $k=1,2,\dots,n-1$.  The following is~clear:

\begin{corollary}\label{cor:ktfanofannulus}
The balanced Knutson--Tao fan $\fan(\ann,\idealtriang)\subset\fan\times\fan$ for the annulus consists of the points $(\sum_{k=1}^{n-1} m_k t_k^L,\sum_{k=1}^{n-1} m_k t_k^R)$ for $m_k\in\nnegZ$.  In particular, the natural projection of $\fan(\ann,\idealtriang)$ to either component $\fan$ is a bijection.  So, by Proposition {\upshape\ref{thm:KTfan}} there is an inclusion of $\fan(\ann,\idealtriang)$ as a sub-monoid of a single real $(n-1)$-dimensional cone in $\R_{\geq0}^{n-1}$ spanned over $\R_{\geq0}$ by the $n-1$ Hilbert basis elements $(t_k^L,t_k^R)$ of $\fan(\ann,\idealtriang)$.  \qed
\end{corollary}

Note (Proposition \ref{prop:second-main-result-independenceforstatedskeins}, proof of Proposition \ref{prop:second-main-result-independence}) that, for $\qroot\in\sC\setminus\badset$, the points of $\fan(\ann,\idealtriang)$ are precisely the pairs of left and right highest degrees of the quantum traces $\qtr\circ\incl((\unbased{B}_{1})^{m_1}(\unbased{B}_{2})^{m_2}\cdots(\unbased{B}_{n-1})^{m_{n-1}})\in\qtorus{\ttriang^L}\otimes\qtorus{\ttriang^R}$ of the monomial basis webs for the skein algebra $\usk\cong\sk$ appearing in the main theorem.  

\begin{remark}
We view Theorem \ref{thm:maintheorem} and Corollary \ref{cor:ktfanofannulus} together as a quantum topological manifestation, in the setting of the annulus, of quantum Fock--Goncharov duality, rather, quantum Fock--Goncharov--Shen duality \cite{davisonMR4337972, fockMR2233852, fockMR2567745, goncharovMR3418241, goncharov2022quantum, grossMR3758151}; see also \cite{douglasMR4685684, douglas2022tropical, kim2022rm, shen2023intersections} for a study of the $n=3$ case for any punctured surface.  

For a punctured surface $\surf$, without boundary for simplicity, Fock--Goncharov associate two dual moduli spaces $\mcal{X}_{\pngl}$ and $\mcal{A}_{\nsl}$.  These are positive spaces, so their points with respect to a  semi-field are defined.  For example, for $n=2$ then the points $\mcal{X}_{\ptwogl}(\R_{>0})$ and $\mcal{A}_{\twosl}(\R_{>0})$ over $\R_{>0}$ recover the classical Teichm\"{u}ller space, rather, enhanced Teichm\"{u}ller space, and Penner's decorated Teichm\"{u}ller space of the surface, respectively, and recover so-called higher Teichm\"{u}ller spaces when $n=n$.  For the  semi-field of tropical integers, denoted $\Z^t$, where multiplication (resp. addition) in $\Z^t$ is the usual addition (resp. max) operation in $\Z$, the points $\mcal{X}_{\pngl}(\Z^t)$ and $\mcal{A}_{\nsl}(\Z^t)$ are also defined.  To work in coordinates, one needs to choose additional topological data such as an ideal triangulation $\idealtriang$ of $\surf$.

Classical Fock--Goncharov duality, so when $\qroot=1$, asserts that the tropical integer points $\mcal{A}_{\nsl}(\Z^t)$ index the elements of a `canonical' basis of the coordinate ring $\mcal{O}(\mcal{X}_{\pngl})$, and vice-versa.  This basis should satisfy, for example, strong positivity properties.  Fock--Goncharov--Shen duality asserts that the positive tropical integer points $\mcal{A}^+_{\pngl}(\Z^t)$, where the positivity is taken with respect to the Goncharov--Shen potential, index the elements of a canonical basis of the coordinate ring $\mcal{O}(\mcal{R}_{\nsl})$ of the character variety.   

There are also quantum versions of this duality.  Since the skein algebra $\sk(\surf)$ is a quantization of the coordinate ring $\mcal{O}(\mcal{R}_{\nsl})$ of the character variety, Fock--Goncharov--Shen duality suggests that the positive tropical integer points $\mcal{A}^+_{\pngl}(\Z^t)$ should index the elements of a canonical basis of the skein algebra, in particular a canonical topological basis of $\mcal{O}(\mcal{R}_{\nsl})$ when $\qroot=1$.  Indeed, following Goncharov--Shen's lead, upon choosing an ideal triangulation $\idealtriang$ of $\surf$ we can think of the corresponding balanced Knutson--Tao fan  as a model for $\mcal{A}^+_{\pngl}(\Z^t)$.  When the surface $\surf=\ann$ is the annulus with triangulation $\idealtriang$ as above, then this model for $\mcal{A}^+_{\pngl}(\Z^t)$, rather, the subset constrained to be zero on the boundary, is precisely the balanced Knutson--Tao fan $\fan(\ann,\idealtriang)$ appearing in Corollary \ref{cor:ktfanofannulus}, whose points index the monomial basis elements of the skein algebra $\sk(\ann)$ appearing in Theorem \ref{thm:maintheorem} via their highest degrees with respect to the quantum trace map $\qtr$.  That these points appear as the highest polynomial degrees of the basis elements, as detected by the quantum trace map, is another  indication of the tropical geometric nature of this duality. 
\end{remark}

\bibliographystyle{plain} 
\bibliography{references.bib}
\end{document}